\documentclass[10pt,twocolumn,letterpaper]{article}

\usepackage{cvpr}
\usepackage{times}
\usepackage{epsfig}
\usepackage{graphicx}
\usepackage{amsmath}
\usepackage{amssymb}
\usepackage{amsthm}
\usepackage{tikz}
\usepackage{pgfplots}
\usepackage{bm}
\usepackage[utf8x]{inputenc}
\usetikzlibrary{plotmarks}

\usepackage[pagebackref=true,breaklinks=true,letterpaper=true,colorlinks,bookmarks=false]{hyperref}

 \cvprfinalcopy 

\def\cvprPaperID{512} 

\def\reg{r}
\def\A{{\mathcal{A}}}
\newcommand{\rank}{\mathsf{rank}}
\renewcommand{\S}{\mathcal{S}}
\newcommand{\Q}{\mathcal{Q}}

\newcommand{\x}{\textbf{x}}
\renewcommand{\b}{\textbf{b}}
\newcommand{\R}{\mathbb{R}}
\newcommand{\y}{\textbf{y}}
\newcommand{\z}{\textbf{z}}
\renewcommand{\v}{\textbf{v}}
\renewcommand{\u}{\textbf{u}}
\renewcommand{\t}{\textbf{t}}

\newcommand{\p}{\textbf{p}}
\newcommand{\q}{\textbf{q}}
\newcommand{\fro}[1]{\left\| #1\right\|^2}

\newcommand{\prox}{\mathsf{prox}}

\def\sign{{\text{sign}}}
\def\bfsigma{{\bm \sigma}}

\DeclareMathOperator*{\argmin}{arg\,min}

\newtheorem{theorem}{Theorem}
\newtheorem{lemma}{Lemma}
\newtheorem{proposition}{Proposition}

\ifcvprfinal\fi
\begin{document}

\title{Bias Reduction in Compressed Sensing}

\def\ww{5mm}
\author{Carl Olsson$^{1,2}$ \hspace{\ww} Marcus Carlsson$^2$ \hspace{\ww} Daniele Gerosa$^2$\\[0.2 cm]
	\begin{minipage}[c]{0.4\textwidth}
		\centering
		${}^1$Department of Electrical Engineering\\
		Chalmers University of Technology 
	\end{minipage}
	\begin{minipage}[c]{0.4\textwidth}
		\centering
		${}^2$Centre for Mathematical Sciences\\
		Lund University 
	\end{minipage}
	\\[0.4cm]
{\tt\small caols@chalmers.se \hspace{\ww} mc@maths.lth.se \hspace{\ww} daniele.gerosa@math.lu.se}
}

\maketitle

\begin{abstract}
	Sparsity and rank functions are important ways of regularizing under-determined linear systems. Optimization of the resulting formulations is made difficult since both these penalties are non-convex and discontinuous. The most common remedy is to instead use the $\ell^1$- and nuclear-norms. While these are convex and can therefore be reliably optimized they suffer from a shrinking bias that degrades the solution quality in the presence of noise. 
	
	In this paper we combine recently developed bias free non-convex alternatives with the nuclear- and $\ell^1-$penalties. This reduces bias and still enables reliable optimization properties.
	We develop an efficient minimization scheme using derived proximal operators and evaluate the method on several real and synthetic computer vision applications with promising results. 
\end{abstract}

\section{Introduction and Background}

Sparsity and rank penalties are common tools for regularizing ill posed linear problems. The sparsity regularized problem is often formulated as
\begin{equation}
\min_{\x} \mu \|\x\|_0 + \|A\x-\b\|^2,
\label{eq:sparsity}
\end{equation}
where $\|\x\|_0$ is the number of non-zero elements of $\x$.
Optimization of \eqref{eq:sparsity} is difficult since the term $\|\x\|_0$ is non-convex and discontinuous. Moreover, the sought solutions should be sparse and therefore the minimizer is typically in the vicinity of discontinuities. A common practice is to replace $\|\x\|_0$ with the $\ell^1$-norm, resulting in the convex relaxation
\begin{equation}
\min_{\x} \lambda \|\x\|_1 + \|A\x-\b\|^2.
\label{eq:1normrelax}
\end{equation}
The seminal works \cite{tropp-2015,tropp-2006,candes2006stable,candes-tao-2006,donoho-elad-2002} gave performance guarantees for this approach.
The notion of Restricted Isometry Property (RIP), which was introduced in \cite{candes2006stable,candes-tao-2006}, states that $A$ obeys a RIP if
\begin{equation}
(1-\delta_K)\|\x\|^2 \leq \|A \x\|^2 \leq (1+\delta_K)\|\x\|^2,
\label{eq:vectorRIP}
\end{equation}
for all vectors $\x$ with $\|\x\|_0 \leq K$.
In \cite{candes2006stable}, Cand\`{e}s, Romberg and Tao proved the surprising result that, if $A$ fulfills $\delta_{3K}+3\delta_{4K} < 2$, given a sparse vector $\x_0$ and a measurement  $\b=A\x_0+\epsilon$ where $\epsilon$ is Gaussian noise, solving \eqref{eq:1normrelax} yields a vector $\hat \x$ that satisfies
\begin{equation}\label{crt}
\|\hat \x-\x_0\|<C_K\|\epsilon\|,
\end{equation}
where $C_K$ is a constant. While \cite{candes2006stable} argued that it is impossible to beat a linear noise dependence, it has been observed that \eqref{eq:1normrelax} suffers from a shrinking bias  \cite{fan2001variable,mazumder2011sparsenet}. Since the $\ell^1$ term not only has the (desired) effect of forcing many entries in $\x$ to 0, but also the (undesired) effect of diminishing the size of the non-zero entries. This has led to a large amount of non-convex alternatives to replace the $\ell^1$-penalty, see e.g. \cite{bredies2015minimization,blumensath2008iterative,pan2015relaxed,zou2008one,wang2014optimal,loh2013regularized,fan2014strong,zhang2012general,loh2017support,candes2008enhancing,breheny2011coordinate}. Typically these come without global convergence guarantees.

In \cite{carlsson2016convexification,carlsson2018unbiased,olsson-etal-iccv-2017} a non-convex alternative that provides optimality guarantees is studied. These papers propose to replace the term $\mu \|\x\|_1$ with $\Q_2(\mu \|\cdot\|_0)(\x)$, where $\Q_\gamma(f)$ is the so called quadratic envelope of $f$, which is defined so that
$\Q_\gamma(f)+\frac{\gamma}{2}\|\x\|^2$ is the convex envelope of $f(\x)+ \frac{\gamma}{2}\|\x\|^2$ \cite{carlsson2016convexification}.
For $f(\x) = \mu\|\x\|_0$ this results in the objective
\begin{equation}
\sum_i \mu-\max(\sqrt{\mu}-|x_i|,0)^2 + \|A\x-\b\|^2.
\label{eq:myrelax}
\end{equation}
The resulting relaxation is non-convex.
However, it was shown in \cite{carlsson2018unbiased,olsson-etal-iccv-2017} that
if a RIP constraint holds then sparse minimizers of this formulation are generally unique (and globally optimal) even though there may exist additional non-sparse local minima. Furhtermore, under the slightly more general RLIP condition,
which is basically the lower bound in \eqref{eq:vectorRIP}, \cite{carlsson2018unbiased} showed that an estimate similar to \eqref{crt} but with a much smaller constant exists. Additionally, in contrast to \eqref{eq:1normrelax} the global optimizer of \eqref{eq:myrelax} is the so called "oracle solution" \cite{candes2006robust}, which is what we get if we minimize $\|A\x-\b\|^2$ over the "true" support of $\x_0$.

Rank regularization methods are largely analogous to the sparsity approaches. If $\A:\mathbb{R}^{m\times n} \rightarrow \mathbb{R}^p$ is a linear operator we are seeking to minimize
\begin{equation}
\mu \rank(X) + \|\A X-\b\|_F^2.
\label{eq:matrixproblem}
\end{equation}
The matrix version of the RIP constraint
\begin{equation}
(1-\delta_r)\|X\|^2_F \leq \|\A X\|^2 \leq (1+\delta_r)\|X\|_F^2,
\label{eq:matrixRIP}
\end{equation}
for all matrices $X$ with $\rank(X)\leq r$, was introduced in \cite{recht-etal-siam-2010}. Similar to the sparsity setting \cite{recht-etal-siam-2010,candes-etal-acm-2011} propose to
replace $\rank(X)$ with the convex nuclear norm $\|X\|_* = \sum_i \sigma_i(X)$, where $\sigma_i(X)$, $i=1,...,N$ are the singular values of $X$.
Since then a number of generalizations that give performance guarantees for the nuclear norm relaxation have appeared, e.g. \cite{recht-etal-siam-2010,oymak2011simplified,candes-etal-acm-2011,candes2009exact}.
The nuclear norm penalizes both small and large singular values and therefore exhibits the same kind of shrinking bias as the $\ell^1$-norm.
To address this non-convex alternatives that are locally optimized have been shown to improve performance \cite{oymak-etal-2015,mohan2010iterative,hu-etal-pami-2013,oh-etal-pami-2016}. In \cite{larsson-olsson-ijcv-2016} it was shown how to compute the convex envelope and proximal operator of
\begin{equation}
g(\rank X)+\|X-X_0\|_F^2,
\end{equation}
where $g$ is a non-decreasing convex function.
Setting $g(\rank X ) = \mu\rank X$ shows that $\Q_2(\mu \rank)(X) = \sum_i \mu-\max(\sqrt{\mu}-\sigma_i(X)|,0)^2$.
Therefore \cite{olsson-etal-iccv-2017} proposed solving
\begin{equation}
\sum_i \mu-\max(\sqrt{\mu}-\sigma_i(X),0)^2 + \|\A X-b\|^2,
\label{eq:Smatrixrelax}
\end{equation}
and gave optimality conditions under RIP \eqref{eq:matrixRIP}, similar to the sparsity case.

While the non-convex relaxations \eqref{eq:myrelax} and \eqref{eq:Smatrixrelax} provide unbiased alternatives to the $\ell^1-$/nuclear- norms and are guaranteed \cite{olsson-etal-iccv-2017,carlsson2018unbiased} to only have one sparse/low-rank stationary point when a sufficiently strong RLIP holds,
it is clear that there can still be poor local minimizers.
For example, let $\x_h$ be a dense vector from the nullspace of $A$ and
$\x_p$ a minimizer of $\|A \x - \b\|$. Then by 
rescaling $\x_h$ so that all the elements of $\x_p+\x_h$ have magnitude strictly larger than $\sqrt{\mu}$ we obtain a vector that 
minimizes the data fit while the regularization $\Q_\gamma(\mu\|\cdot\|_0)$ is (locally) constant around it. This construction also provides local minimizers for the original formulations \eqref{eq:matrixproblem}, \eqref{eq:sparsity} and any other unbiased formulation that leaves elements/singular values larger than some threshold unpenalized. In this paper we remedy this by introducing a small penalty for large singular values. We study relaxations of 
\begin{equation}
\mu \|\x\|_0 + \lambda\|\x\|_1 + \|A\x-b\|_F^2,
\label{eq:cardandl1}
\end{equation}
and
\begin{equation}
\mu \rank (X) + \lambda\|X\|_* + \|\A X-b\|_F^2.
\label{eq:rankandnuclear}
\end{equation}
for sparsity and rank regularization respectively.
We propose to solve these by replacing the relaxation terms with their quadratic envelopes $\Q_2(\mu\|\cdot\|_0+\lambda\|\cdot\|_1)$ and $\Q_2(\mu\rank+\lambda\|\cdot\|_*)$.
Our formulation can be seen as a trade-of between small bias and improved optimization properties. While the terms $\gamma \|\x\|_1$ and $\gamma \|X\|_*$ introduce a bias to small solutions they also increase the convergence basin for cases where the RLIP is not strong enough to ensure global convergence of \eqref{eq:myrelax} and \eqref{eq:Smatrixrelax}.

Simple optimization is often related to good modeling.
Adding a weak shrinking factor may also make sense from a modeling perspective for certain applications. In this paper we exemplify with non-rigid structure form motion (NRSfM). Here each non-zero singular value corresponds to a mode of deformation. When choosing a weaker $\mu$ (larger rank) in order to capture all fine deformations the resulting problem is often ill posed due to unobserved depths. As noted in \cite{olsson-etal-iccv-2017} this may result in a large difference to the true reconstruction despite good data fit. The addition of the $\lambda \|X\|_*$ allows us to separately incorporate a variable bias restricting the size of the  deformations, which regularizes the problem further, see Section~\ref{sec:nonrigid}.

The main contributions of this paper are
\begin{itemize}
	\item We present a class of new regularizers that leverage the benefits of previous convex and unbiased formulations. 
	
	\item We characterize the stationary points of the proposed formulation, and give theoretical results that guarantee the uniqueness of a sparse stationary point.
	
	\item We show how to compute proximal operators of our regularization enabling fast optimization via splitting methods such as $ADMM$ \cite{boyd-etal-ftml-2011}.
	
	\item We show that our new formulations generate better solutions in cases where a weak or no RIP holds.
\end{itemize}

\section{Relaxations and Shrinking Bias}
In this section we will study properties of our proposed relaxations of \eqref{eq:cardandl1} and \eqref{eq:rankandnuclear}. We will present our results in the context of the vector case \eqref{eq:cardandl1}. The corresponding matrix versions follow by applying the regularization term to the singular values, with similar almost identical proofs. Our first theorem (which is proven in Appendix~\ref{app:A}) shows that adding the term $\|\cdot\|_1$ before or after taking the quadratic envelop makes no difference.

\begin{theorem}\label{thm:convenv}Let \( f:\mathbb{R}^d  \to \mathbb{R}   \) be a lower semicontinuous sign-invariant function such that \( f(\mathbf{0})=0  \) and \( f(\x + \y)\ge f(\x)  \) for every \( \x, \y \in \mathbb{R}^d _+  \). Then
\begin{equation}
\Q_2(f + \lambda \| \cdot\|_1 )(\x)= \Q_2(f)(\x) + \lambda \| \x \|_1    \end{equation}
for every \( \x \in \mathbb{R}^d \).
\end{theorem}

In view of the above it is clear that $\Q_2(\gamma\|\cdot\|_0 + \lambda \| \cdot\|_1 ) = \reg_{\lambda,\mu}$,
where 
\begin{equation}
\reg_{\mu,\lambda}(\x) = \sum_i \left(\mu-\max(\sqrt{\mu}-|x_i|,0)^2\right) + \lambda \|\x\|_1.
\label{eq:ourreg}
\end{equation}
We therefore propose to minimize the objective
\begin{equation}
r_{\mu,\lambda}(\x)+\|A\x-\b\|^2.
\label{eq:newrelax}
\end{equation}
Note that $\reg_{\mu,\lambda}(\x)+\|A\x -\b\|^2$ reduces to \eqref{eq:1normrelax} if $\mu=0$ and \eqref{eq:myrelax} if $\lambda=0$.

Figure~\ref{fig:fun} shows an illustration of $\reg_{\mu,\lambda}$ for a couple of different values of $\lambda$. When $\lambda=0$ the function is constant for values larger than $\sqrt{\mu} = 1$. Therefore large elements give zero gradients which can result in local algorithms getting stuck in poor local minimizers. Note that if RIP holds the results of \cite{carlsson2018unbiased,olsson-etal-iccv-2017} show that such non-global minimizers can not be sparse under moderate noise conditions, and therefore this situation could in practice be detected. However, for a general problem instance increasing $\lambda$ makes the regularizer closer to being convex, which as we show in Section~\ref{sec:exp}, increases its convergence basin.
\begin{figure}[htb]
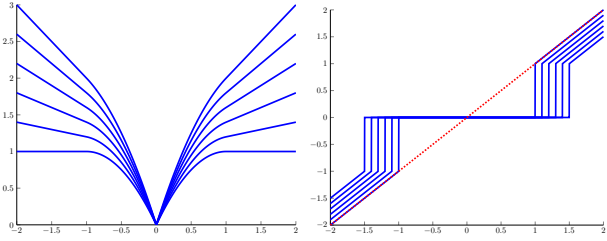

\resizebox{40mm}{!}{\input{reg.tex}}
\resizebox{40mm}{!}{\input{solset.tex}}
\caption{Left: The function \eqref{eq:ourreg} for $\mu=1$ and $\lambda = 0,0.2,0.4,...,1$.
Right: Illustration of \eqref{eq:solset} for the same $\mu$ and $\lambda$. The graphs show $z_i$ (x-axis) versus $x^*_i$ (y-axis). If $z_i = \sqrt{\mu}+\frac{\lambda}{2}$ all $x^*_i \in [0,\sqrt{\mu}]$ are optimal.} 
\label{fig:fun}
\end{figure}

To characterize the kind of shrinking bias that we can expect from this family of relaxations we now consider the problem of minimizing 
\begin{equation}
\min_\x \reg_{\mu,\lambda}(\x)+\|\x - \z\|^2.
\label{eq:localS}
\end{equation}
The minimization is separable in the elements of $\x$ and therefore we are able to solve it by considering
\begin{equation}
\min_{x_i} \mu - (\max(\sqrt{\mu}-|x_i|,0))^2 +\lambda |x_i| + (x_i-z'_i)^2.
\end{equation} 
This is a one dimensional problem that can easily be solved by computing stationary points, see Appendix~\ref{app:B} for details.
This resulting minimizer is given by
\begin{equation}
x^*_i \in 
\begin{cases}
\{z_i -\sign(z_i)\frac{\lambda}{2}\} & |z_i| > \sqrt{\mu}+\frac{\lambda}{2} \\
[0, \sqrt{\mu}]\sign(z_i) & |z_i| = \sqrt{\mu}+\frac{\lambda}{2} \\
\{0\} & |z_i| < \sqrt{\mu}+\frac{\lambda}{2}
\end{cases}.
\label{eq:solset}
\end{equation}
Figure~\ref{fig:fun} illustrates the solution set \eqref{eq:solset} for $\mu=1$ and $\lambda=0,0.2,0.4,...,1$.  The shrinking bias comes from the subtraction of $\frac{\lambda}{2}$ from the magnitude of the large elements. Since we would ideally like these to remain unchanged it is essential to keep $\lambda$ small. On the other hand a larger $\lambda$ makes the regularization function $\reg_{\mu,\lambda}$ "more convex" which simplifies optimization.

We conclude this section with a simple 2D-illustration of the general principle. 
\begin{figure}[htb]
	\begin{center}
		\includegraphics[width=39mm]{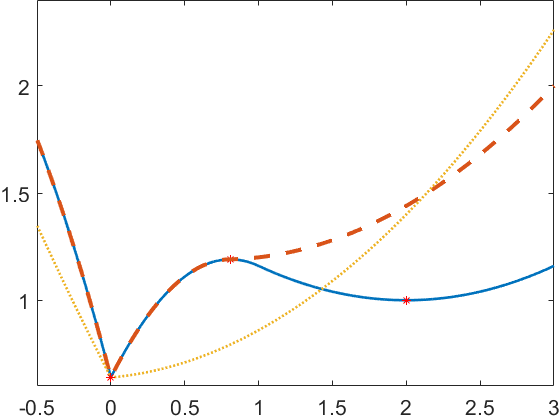}
		\includegraphics[width=40mm]{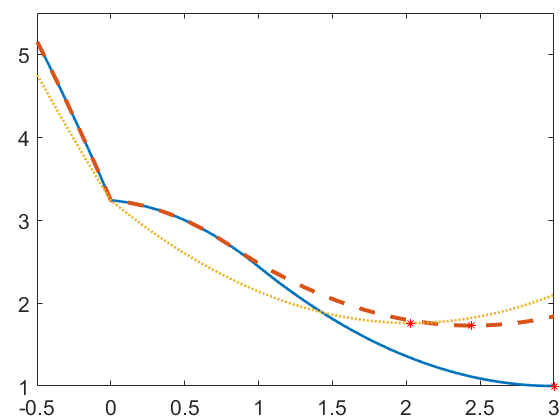}
	\end{center}
	\caption{The residual errors of a two dimensional illustration of the different regularizes.}
	\label{fig:twoDillustration}
\end{figure}
Figure~\ref{fig:twoDillustration} shows the error residuals of 
$\reg_{\mu,\lambda}(\x)+\|A\x-\b\|^2$ for a two dimensional problem with 
\begin{equation}
A = \left(\begin{matrix}
0.4 & 0 \\
0 & 0.6
\end{matrix}\right) \quad \text{and} \quad 
b = \left(\begin{matrix}
0.8 \\
1.8
\end{matrix}\right).
\end{equation}
The blue curves show the two error residuals for $\mu=1$ and $\lambda=0$.
It is easy to verify that the problem has two local minimizers $\x=(2,3)$ and $\x=(0,3)$ (which is also global).
These points (and in addition $(0,0)$ and $(2,0)$) are also local minima to \eqref{eq:sparsity} with $\mu=1$.
The yellow curve shows the effect of using the convex $\ell_1$ formulation \eqref{eq:1normrelax}, with $\lambda=0.7$. Here we have used the smallest possible $\lambda$ so that the optimum of the left residual is $0$ while the right one is non-zero. The resulting solution $\x=(0,2)$ has the correct support however, the magnitude of the non-zero element is reduced from $3$ to $2$ due to the shrinking bias. 
With our approach it is possible to chose an objective which has less bias but still a single local minimizer. Setting $\mu=0.7$ and $\lambda=0.4$ gives the red dashed curves with optimal point $\x \approx (0,2.5)$. 

\section{Oracle Solutions}

For sparsity problems the so called oracle solution \cite{candes2006stable} is what we would obtain if we somehow knew the "true" support of the solution and we were to solve the least squares problem over the non-zero entries of $\x$.
We will use the notation $A_S$ to denote the matrix which has the same entries as $A$ in the columns indexed by $S$ and zeros otherwise. Similarly $A_{\overline{S}}$ have zeros in the columns in $S$ and 
$\x_S$ is a vector with zero elements in $\overline{S}$.
The oracle solution is then 
\begin{equation}
\min_{x_S} \|A_S \x_S - \b \|^2.
\end{equation}
Cand\'es \etal \cite{candes2006stable} showed that under RIP the solution \eqref{eq:1normrelax} approximates the oracle solution. 
In \cite{carlsson2018unbiased} it was shown that under similar conditions \eqref{eq:myrelax} gives exactly the oracle solution.
In this section we will show that our relaxation solves a similar $\ell_1$-regularized least squares problem. 

Suppose $\epsilon = A\x_0 - \b$ and let $S$ be the set of non-zero indexes of $x_0$.
We refer to the regularized oracle solution as the minimizer of 
\begin{equation}
\min_{\x_S} \lambda\|\x_S\|_1+\|A_S\x_S - \b\|^2.
\label{eq:oracle}
\end{equation}
For $\lambda = 0$ this is the least squares solution over the correct support, which is the best we can hope for in the presence of Gaussian noise. For a non-zero $\lambda$ the $\ell^1$ norm modifies the solution by adding a shrinking bias. Note however that in contrast to standard approaches where the $\ell^1$-norm is used to promote sparsity via soft thresholding, here $\x_S$ is already sparse. The value of $\lambda$ is intended to be small as in \eqref{eq:newrelax} where the $\ell^1$-norm is used to increase the convergence basin of \eqref{eq:myrelax}.

\begin{theorem} If $\mu$ is selected such that the solution $\x_S$ of \eqref{eq:oracle} fulfills $|x_{Si}| \notin (0,\sqrt{\mu})$ and the residual errors $\epsilon = A_S\x_S - \b$ are bounded in the sense that 
\begin{equation}
\|A_{\overline{S}}^T \epsilon \|_\infty < \sqrt{\mu},
\end{equation}
then $\x_S$ is a {\bf stationary point} of \eqref{eq:newrelax}.
\end{theorem}

\begin{proof}
We first note by differentiating that the solution of \eqref{eq:oracle} fulfills
\begin{equation}
\lambda \v + 2 A_S^T A_S \x_S -2 A_S^T \b = 0, 
\end{equation}
where $\v \in \partial \|\x_S\|_1$. 
The subdifferential of $\|\cdot\|_1$ can be computed element-wise and is given by
\begin{equation}
\partial |x_{Si}| = \begin{cases}
\sign(x_{Si}) & x_{Si} \neq 0 \\
[-1,1] & x_{Si} = 0 \\
\end{cases}.
\label{subdiffg0}
\end{equation}
Note that since the columns in $\overline{S}$ of $A_S$ are zero it is clear that so are the elements in $\overline{S}$ of $\v$, that is, $\v=\v_S$.

Let $g_\lambda(\x) = \reg_{\mu,\lambda}(\x)+\|\x\|^2$ and $h(\x) = -\|\x\|^2+\|A\x-\b\|^2$. 
Note that $g_\lambda(\x) = g_0(\x)+\lambda\|\x\|_1$ is convex with a well defined subdifferential. 
To show that $\x_S$ is stationary in \eqref{eq:newrelax} we need to show that there is a vector $\u_1\in \partial g_0(\x_S)$ and $\u_2 \in \partial\|\x_S \|_1$ such that
\begin{equation}
\u_1+\lambda \u_2 = -\nabla h(\x_S) = 2(I-A^T A)\x_S + 2A^T \b.
\label{eq:u1u2}
\end{equation}
The subdifferential of $g_0$ can be evaluated element-wise giving
\begin{equation}
\partial g_0(x_i) = 
\begin{cases}
\{2x_i\} & |x_i| \geq \sqrt{\mu} \\
\{2\sqrt{\mu}\sign(x_i)\} & 0 < |x_i| \leq \sqrt{\mu} \\
[-2\sqrt{\mu},2\sqrt{\mu}] & x_i=0
\end{cases}.\label{eq:g0subdiff}
\end{equation}

Since $\x_S$ has support in $S$ we have $A\x_S = A_S\x_S$. Furthermore, since $A=A_S+A_{\overline{S}}$ equation \eqref{eq:u1u2} can be written
\begin{equation}
\u_1+\lambda \u_2 = 2\x_S - 2(A_{\overline{S}}^T A_S + A_S^T A_S) \x_S  + 2(A_{\overline{S}}^T+ A^T_S)\b.
\end{equation}
It is clear if we select $\u_2 = \v_S$ this reduces to 
\begin{equation}
\u_1 = 2\x_S + A_{\overline{S}}^T (A_S \x_S- \b) = 2\x_S + 2A_{\overline{S}}^T \epsilon.
\end{equation}
On $S$ we have $\u_1 = 2 \x_S$. Since the non-zero elements of $\x_S$ are larger than $\sqrt{\mu}$ the first case of \eqref{eq:g0subdiff} holds on $S$.
Similarly, on $\overline{S}$ we have $\u_1 = 2A_{\overline{S}}^T \epsilon$ with elements smaller than $2\sqrt{\mu}$. 
Since $\x_S$ is zero on $\overline{S}$ case 3 of \eqref{eq:g0subdiff} holds here, which shows that
$\u_1 \in \partial g_0(\x_S)$. 
\end{proof}

Whether $\x_S$ is the global optimum or not depends on the problem instance. In the following sections we will show that under a sufficiently strong RIP it is the sparsest possible stationary point.

\section{Separation of Stationary Points}
In this section we study the stationary points of the objective function
\eqref{eq:newrelax} under the assumption that $A$ fulfills the RIP condition \eqref{eq:vectorRIP}
for all $\x$ with $\|\x\|_0 \leq K$.
We will extend the results of \cite{carlsson2018unbiased,olsson-etal-iccv-2017} to our class of functionals. Specifically, we show that under some technical conditions two stationary points $\x'$ and $\x''$ have to be separated by $\|\x''-\x'\|_0 > K$. From a practical point of view this means that if we find a stationary point with $\|\x'\|_0 \leq \frac{K}{2}$ we can be certain that this is the sparsest one possible.

\subsection{Stationary Points and Local Approximation}
We will first characterize a stationary point as being a thresholded version of a noisy vector $\z$ which depends on the data. Let $g_\lambda$ and $h$ be defined as in the previous section.

\begin{lemma}\label{lemma:lowrankstatpt}
	If $\z' =  (I -A^T A)\x'+A^T\b$ the point $\x'$ is stationary in \eqref{eq:newrelax}
	if and only if $2\z' \in \partial g_\lambda(\x')$ and if and only if
	\begin{equation}
	\x' \in \argmin_\x \reg_{\mu,\lambda}(\x) + \|\x -  \z'\|^2. \label{eq:localS}
	\end{equation}
\end{lemma}
\begin{proof}
	By differentiating $g_\lambda(\x)+h(\x)$ we see that $\x'$ is stationary in \eqref{eq:newrelax} if and only if
	$2\z' \in \partial g_\lambda(\x')$. Similarly, differentiating \eqref{eq:localS} we see that $\x'$ is stationary in \eqref{eq:localS} if and only if $2\z' \in \partial g_\lambda(\x')$.
	Now recall that by definition \eqref{eq:localS} is convex and therefore $\x'$ being stationary is equivalent to solving \eqref{eq:localS}.
\end{proof}

We will use properties of the vector $\z'$ to establish conditions that ensure that $\x'$ is the sparsest possible stationary point to \eqref{eq:newrelax}. The overall idea which follows \cite{olsson-etal-iccv-2017,carlsson2016convexification} is to show that subdifferential $\partial g_\lambda$ grows faster than $-\nabla h$ and therefore we can only have $-\nabla h(\x) \in \partial g_\lambda(\x)$ in one (sparse) point.
This requires an estimate of the growth of the subgradients of $g_\lambda$ which we now present.

The function $g_\lambda$ is separable and can be evaluated separately for each element of $\x$.
The subdifferential $\partial g_\lambda(x)$ can be written
\begin{equation}
\partial g_\lambda(x) = 
\begin{cases}
\{2x + \lambda \sign(x)\} & |x| \geq \sqrt{\mu} \\
\{(2\sqrt{\mu}+\lambda)\sign(x)\} & 0 < |x| \leq \sqrt{\mu} \\
[-2\sqrt{\mu}-\lambda,2\sqrt{\mu}+\lambda] & x=0
\end{cases}.\label{eq:gsubdiff}
\end{equation}
\begin{figure}[htb]
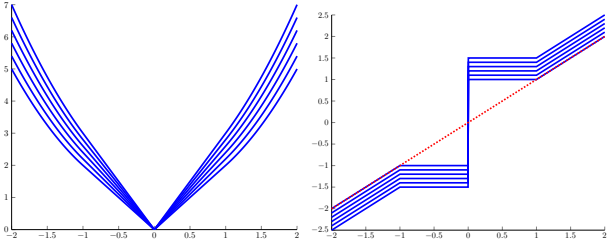

	\begin{center}
		\resizebox{40mm}{!}{\input{gfun.tex}}
		\resizebox{40mm}{!}{\input{gsubdiff.tex}}
	\end{center}
	\caption{The function $g(x)$ (left) and the subdifferential $\partial g_\lambda(x)/2$ (right) for $\mu=1$ and $\lambda=0,0.2,0.4,...,1$. (For comparison we also plot the red dotted curve $y=x$.}
	\label{fig:gfun}
\end{figure}
Figure~\ref{fig:gfun} shows the function $g_\lambda$ and $\partial g_\lambda$. The parameter $\lambda$ adds a constant offset to the positive values of $\partial g_\lambda(x)$ and subtracts the same value for all negative values.
It is clear from Figure~\ref{fig:gfun} that in $(-\infty,-\sqrt{\mu}]$ and $[\sqrt{\mu},\infty)$ 
the subdifferential contains a single element. In addition for any two elements $x'',x'$ in one of these intervals we have
\begin{equation}
\langle \partial g_\lambda(x'')-\partial g_\lambda(x'),x''-x'\rangle = |x''-x'|^2.
\end{equation}
For the other parts the subdifferential grows less. To ensure a certain growth we need to add some assumptions on the subdifferential which is done in the following result (which is proven in Appendix~\ref{app:C}).

\begin{lemma}\label{lemma:subgradbnd}
	Assume that $2\z' \in \partial g_\lambda(\x')$. If the elements $z'_i$ fulfill
	$|z'_i| \notin [(1-\delta_K)\sqrt{\mu}+\frac{\lambda}{2},\frac{\sqrt{\mu}}{1-\delta_K}+\frac{\lambda}{2}]$ for every $i$, then for any $\z'$ with $2\z'' \in \partial g_\lambda(\x'+\v)$ we have
	\begin{equation}
	\langle\z''-\z',\v\rangle >  \delta_K \|\v\|^2,
	\label{eq:subdiffest}
	\end{equation}
	as long as $\|\v\| \neq 0$.
\end{lemma}

Note that a similar estimate for the $\nabla h$. Differentiating $h$ gives
\begin{equation}
\nabla h(\x) = 2(I-A^TA)\x + 2A^T\b.
\end{equation}
Since $\nabla h(\x)$ is linear we get
\begin{equation}
\begin{split}
|\langle \nabla h(\x'+\v)-\nabla h(\x'),\v \rangle| =  \\
|2(\|A\v\|^2-\|\v\|^2)| \leq 2\delta_K\|\v\|^2,
\end{split}
\label{eq:hest}
\end{equation}
if $\|\v\|_0 \leq K$, since RIP holds.

We are now ready to prove that stationary points of~\eqref{eq:newrelax} can be separated in terms of the cardinality of their difference. The result requires that the elements of the vector $\z'$ are not too close to the threshold $\sqrt{\mu}+\frac{\lambda}{2}$. This condition is a natural restriction since the vector $\z$ is related to the noise of the problem \cite{carlsson2018unbiased}, and for very large noise levels we can expect that there will be multiple solutions.

\begin{theorem}\label{thm:statpoint}
	Assume that $\x'$ is a stationary point with $2\z'\in\partial g_\lambda(\x')$ and that each element of $\z'$ fulfills $|z'_i| \notin [(1-\delta_K)\sqrt{\mu}+\frac{\lambda}{2},\frac{\sqrt{\mu}}{1-\delta_K}+\frac{\lambda}{2}]$.
	If $\x''$ is another stationary point then $\|\x''-\x'\|_0 >K $. 
	If, in addition, $\|\x'\|_0 <\frac{K}{2} $ then
	\begin{equation}
	\x' \in \argmin_{\|\x\|_0 \leq \frac{K}{2}} \reg_{\mu,\lambda}(\x)+\|A\x-\b\|^2.
	\label{eq:restricted_min}
	\end{equation}
\end{theorem}
\begin{proof}
	Assume that $\|\x''-\x'\|_0  \leq K$.
	Since $\x'$ is a stationary point we have
	$2\z'+ \nabla h(\x') = 0$,
	where $2\z' \in \partial g_\lambda(\x')$.
	We first show that $2\z''+ \nabla h(\x'') \neq 0$ if $2\z'' \in \partial g_\lambda(\x'')$.
	If $\v = \x''-\x'$ has $\|\v\|_0 \leq K$ then according to Lemma \ref{lemma:subgradbnd} and \eqref{eq:hest} we have
	\begin{equation}
	\begin{split}
	\langle 2 \z'' + \nabla h(\x''), \v \rangle = \\
	2\langle\z''-\z',\v\rangle + \langle\nabla h(\x + \v)-\nabla h(\x), \v\rangle > 0.
	\end{split}
	\label{eq:derest}
	\end{equation}
	Thus $\x''$ cannot be stationary. 

	Since $g_\lambda$ is convex its directional derivative exists and is given by
	\begin{equation}
	g'_{\lambda,\v}(\x) = \max_{2\z' \in \partial g_\lambda(\x)}\langle 2 \z',\v \rangle.
	\end{equation}
	Since $h'_\v(\x) = \langle \nabla h'(\x),\v\rangle$,
	\eqref{eq:derest} shows that $g'_{\lambda,\v}(\x'+t\v) + h'_\v(\x'+t\v) > 0$ for any $t > 0$ if $\|\v\|_0 \leq K$.
	Now, if $\|\x'\|_0 <\frac{K}{2}$ then $g_\lambda+h$ is increasing on every line segment between $\x'$ and any other point 
	$\x''$ with $\|\x''\|_0 \leq \frac{K}{2}$.
\end{proof}

\section{Optimization}
In this section we present a simple algorithm for optimizing objective functions of the type \eqref{eq:newrelax}. We restrict ourselves to sparsity problems since the same approach works for rank regularization with minimal changes.

\subsection{Algorithm}
For optimization we employ the popular ADMM \cite{boyd-etal-ftml-2011} approach. This is a splitting scheme that uses two copies of the $\x$ and enforces them to be equal using dual variables.
The augmented Lagrangian for the problem is
\begin{equation}
L(\x,\y,{\bm \eta}) = \reg_{\mu,\lambda}(\x) + \rho \|\x-\y+{\bm \eta}\|_F^2 + \|A\y-\b\|^2 -\rho \|{\bm \eta}\|^2.
\end{equation}
In each iteration $t$ of ADMM the variable updates are given by
\begin{eqnarray}
\x_{t+1}  = & \argmin_\x \reg_{\mu,\lambda}(\x) + \rho \|\x-\y_t+{\bm \eta}_t\|^2, \label{eq:proxop1} \\
\y_{t+1}  = &\argmin_\y \rho \|\x_{t+1}-\y+{\bm \eta}_t\|^2 + \|A\y-\b\|^2 ,\label{eq:proxop2} \\
{\bm \eta}_{t+1}  = & {\bm \eta}_{t} + \x_{t+1}-\y_{t+1}.
\end{eqnarray}
The update in equation \eqref{eq:proxop2} is a simple least squares minimization with closed form solution. Problem \eqref{eq:proxop1} is solved using the proximal operator of $r_{\mu,\lambda}$ which we present in the next section. We refer to \cite{boyd-etal-ftml-2011} for details on ADMM.

\subsection{The proximal operator}\label{prox}

To solve \eqref{eq:prox1} we can use the proximal operator
\begin{equation}\prox_{\frac{\reg_{\mu,\lambda}}{\rho}}(\y)=\argmin_{\x} \frac{1}{\rho}\reg_{\mu,\lambda}(\x)+\|\x-\y\|^2.
\end{equation}
The following result (which is proven in Appendix~\ref{app:D}) shows that in general the proximal operator of
$\Q_\gamma (f + \lambda \| \cdot \|_1)$ is easy to compute if the proximal operator of $\Q_\gamma(f)$ is known.
\begin{proposition}\label{prop:prox}
	Let \( f:\mathbb{R}^d  \to \mathbb{R}   \) be a lower semicontinuous sign-invariant function such that \( f(\mathbf{0})=0  \) and \( f(\x + \y)\ge f(\x)  \) for every \( \x, \y \in \mathbb{R}^d _+  \). Then \begin{equation}\prox_{ \Q_\gamma (f + \lambda \| \cdot \|_1)/ \rho   } (\y) = \prox_{ \Q_\gamma (f)  / \rho  } ( \prox_{\lambda \|\cdot \|_1 / \rho}  ) (\y)       
	\end{equation}for every \( \y \in \mathbb{R}^d   \).	
\end{proposition}

For our case $f(\x) = \mu\|\x\|_0$ the proximal operator is separable and each element of the vector $\x$ can be treated independently. Adding soft thresholding to the results of \cite{larsson-olsson-ijcv-2016} we get that for $\rho > 1$ the proximal solution is given by
\begin{equation}
z_i =
\begin{cases}
\max(y_i- \lambda/2\rho,0) & y_i \geq 0 \\
\min(y_i+ \lambda/2\rho,0) & y_i \leq 0
\end{cases}\label{eq:prox1}
\end{equation}
and
\begin{equation}
x_i =
\begin{cases}
z_i & \sqrt{\mu} \le |z_i|  \\
\frac{(\rho+1) z_i-\sqrt{\mu}\sign(z_i)}{\rho}  & \frac{\sqrt{\mu}}{\rho+1} \le |z_i| \le \sqrt{\mu}\\
0 &  |z_i| \leq  \frac{\sqrt{\mu}}{\rho+1}  \\
\end{cases}.\label{eq:prox2}
\end{equation}

\section{Experiments}\label{sec:exp}
In this section we test the proposed formulation on a number of real and synthetic experiments. Our focus is to evaluate the proposed method's robustness to local minima and the effects of its shrinking bias.

\subsection{Random Matrices}
In this section we compare the robustness to local minima of the relaxations \eqref{eq:1normrelax}, \eqref{eq:myrelax} and \eqref{eq:newrelax}. Note that \eqref{eq:1normrelax} and \eqref{eq:myrelax} are special cases of \eqref{eq:newrelax}, obtained by letting $\lambda$ or $\mu$ equal $0$ (by Theorem \ref{thm:convenv}).

We generated $A$-matrices of size $100 \times 200$ by drawing the columns from a uniform distribution over the unit sphere $\mathbb{S}^{99}$ in $\R^{100}$, and the vector $\x_0$ was selected to have $10$ random nonzero elements with random magnitudes between $2$ and $4$, resulting in $\|x_0\|\approx 10$. We then computed $\b = A\x_0 + \epsilon$ for different values of random noise with $\|\epsilon\|$ ranging from 0 to 5. For \eqref{eq:myrelax} we used $\mu=1$ and for \eqref{eq:1normrelax} we used $\lambda_{\ell^1}=2\frac{\sqrt{2\log(200)}}{\sqrt{200}}\|\epsilon\|\approx 0.5 \|\epsilon\|$; see \cite{carlsson2018unbiased} for the rationale behind these choices. For \eqref{eq:newrelax} we again chose $\mu=1$ but used $\lambda=\lambda_{\ell^1}/6$. Figure \eqref{fig:hist} plots $\|\x-\x_S\|$ for the estimated $\x$ with the three methods, as a function of $\|\epsilon\|$. Both \eqref{eq:myrelax} and \eqref{eq:newrelax} do better than traditional $\ell^1$ in the entire range, \eqref{eq:myrelax} finds $x_S$ with 100\% accuracy until around $\|\epsilon\|\approx 3$, where \eqref{eq:newrelax} starts to perform better.
This is likely due to the fact that the small $\ell^1$ term helps the (non-convex) method \eqref{eq:newrelax} to not get stuck in local minima. To test this conjecture, we ran the same experiment for 50 iterations for the fixed noise level $\|\epsilon\|=3.5$ and chose as initial point the least squares solution  $\argmin_\x\|A\x-\b\|^2$, which is known to be close to many local minima. The histograms to the right in Figure \ref{fig:hist} show the cardinality of the found solution.
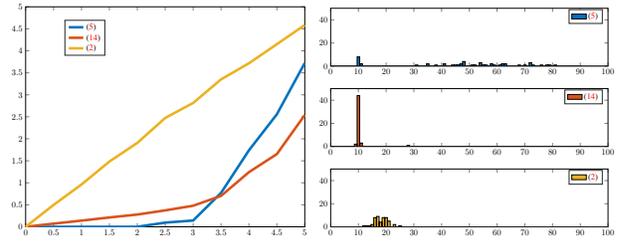
\begin{figure}[htb]
	\begin{center}
		\resizebox*{40mm}{!}{
%
%
\definecolor{mycolor1}{rgb}{0.00000,0.44700,0.74100}%
\definecolor{mycolor2}{rgb}{0.85000,0.32500,0.09800}%
\definecolor{mycolor3}{rgb}{0.92900,0.69400,0.12500}%
\begin{tikzpicture}

\begin{axis}[%
width=4.521in,
height=3.566in,
at={(0.758in,0.481in)},
scale only axis,
xmin=0,
xmax=5,
ymin=0,
ymax=5,
axis background/.style={fill=white},
legend style={at={(0.14,0.779)}, anchor=south west, legend cell align=left, align=left, draw=white!15!black}
]
\addplot [color=mycolor1, line width=3.0pt]
  table[row sep=crcr]{%
0	4.44180156903459e-15\\
0.5	4.56344299588337e-15\\
1	4.78574890944864e-15\\
1.5	4.36886217533249e-15\\
2	4.44225998597334e-15\\
2.5	0.0952243718675539\\
3	0.142375195422821\\
3.5	0.770563074291336\\
4	1.73936687849435\\
4.5	2.55691388996716\\
5	3.7236181868811\\
};
\addlegendentry{\eqref{eq:myrelax}}

\addplot [color=mycolor2, line width=3.0pt]
  table[row sep=crcr]{%
0	4.44180156903459e-15\\
0.5	0.0720319417219271\\
1	0.140480941516429\\
1.5	0.212339898968997\\
2	0.27860999561295\\
2.5	0.373314709810426\\
3	0.478636897015399\\
3.5	0.702200720298089\\
4	1.24206600781905\\
4.5	1.65310260650975\\
5	2.54083289108269\\
};
\addlegendentry{\eqref{eq:newrelax}}

\addplot [color=mycolor3, line width=3.0pt]
  table[row sep=crcr]{%
0	0\\
0.5	0.496801844625762\\
1	0.961325110207102\\
1.5	1.4828189275797\\
2	1.90639196759883\\
2.5	2.47387944248567\\
3	2.81496870245749\\
3.5	3.35107532346059\\
4	3.716631043707\\
4.5	4.15038218861922\\
5	4.5860770273577\\
};
\addlegendentry{\eqref{eq:1normrelax}}

\end{axis}
\end{tikzpicture}%
}
		\resizebox*{40mm}{!}{
%
%
\definecolor{mycolor1}{rgb}{0.00000,0.44700,0.74100}%
\definecolor{mycolor2}{rgb}{0.85000,0.32500,0.09800}%
\definecolor{mycolor3}{rgb}{0.92900,0.69400,0.12500}%
\begin{tikzpicture}

\begin{axis}[%
width=4.521in,
height=0.944in,
at={(0.758in,3.103in)},
scale only axis,
bar shift auto,
xmin=0,
xmax=100,
ymin=0,
ymax=50,
axis background/.style={fill=white},
legend style={legend cell align=left, align=left, draw=white!15!black}
]
\addplot[ybar, bar width=3, fill=mycolor1, draw=black, area legend] table[row sep=crcr] {%
1	0\\
2	0\\
3	0\\
4	0\\
5	0\\
6	0\\
7	0\\
8	0\\
9	0\\
10	8\\
11	2\\
12	0\\
13	0\\
14	0\\
15	0\\
16	0\\
17	0\\
18	0\\
19	0\\
20	0\\
21	0\\
22	0\\
23	0\\
24	0\\
25	0\\
26	0\\
27	0\\
28	0\\
29	0\\
30	0\\
31	1\\
32	0\\
33	0\\
34	0\\
35	2\\
36	0\\
37	0\\
38	1\\
39	0\\
40	0\\
41	2\\
42	0\\
43	0\\
44	1\\
45	1\\
46	1\\
47	2\\
48	4\\
49	0\\
50	0\\
51	1\\
52	1\\
53	0\\
54	3\\
55	1\\
56	1\\
57	0\\
58	2\\
59	1\\
60	0\\
61	1\\
62	2\\
63	2\\
64	0\\
65	0\\
66	0\\
67	0\\
68	1\\
69	0\\
70	1\\
71	0\\
72	3\\
73	1\\
74	0\\
75	0\\
76	1\\
77	0\\
78	1\\
79	1\\
80	0\\
81	1\\
82	0\\
83	0\\
84	0\\
85	0\\
86	0\\
87	0\\
88	0\\
89	0\\
90	0\\
};
\addplot[forget plot, color=white!15!black] table[row sep=crcr] {%
0	0\\
100	0\\
};
\addlegendentry{\eqref{eq:myrelax}}

\end{axis}

\begin{axis}[%
width=4.521in,
height=0.944in,
at={(0.758in,1.792in)},
scale only axis,
bar shift auto,
xmin=0,
xmax=100,
ymin=0,
ymax=50,
axis background/.style={fill=white},
legend style={legend cell align=left, align=left, draw=white!15!black}
]
\addplot[ybar, bar width=3, fill=mycolor2, draw=black, area legend] table[row sep=crcr] {%
1	0\\
2	0\\
3	0\\
4	0\\
5	0\\
6	0\\
7	0\\
8	0\\
9	2\\
10	44\\
11	3\\
12	0\\
13	0\\
14	0\\
15	0\\
16	0\\
17	0\\
18	0\\
19	0\\
20	0\\
21	0\\
22	0\\
23	0\\
24	0\\
25	0\\
26	0\\
27	0\\
28	1\\
29	0\\
30	0\\
31	0\\
32	0\\
33	0\\
34	0\\
35	0\\
36	0\\
37	0\\
38	0\\
39	0\\
40	0\\
41	0\\
42	0\\
43	0\\
44	0\\
45	0\\
46	0\\
47	0\\
48	0\\
49	0\\
50	0\\
51	0\\
52	0\\
53	0\\
54	0\\
55	0\\
56	0\\
57	0\\
58	0\\
59	0\\
60	0\\
61	0\\
62	0\\
63	0\\
64	0\\
65	0\\
66	0\\
67	0\\
68	0\\
69	0\\
70	0\\
71	0\\
72	0\\
73	0\\
74	0\\
75	0\\
76	0\\
77	0\\
78	0\\
79	0\\
80	0\\
81	0\\
82	0\\
83	0\\
84	0\\
85	0\\
86	0\\
87	0\\
88	0\\
89	0\\
90	0\\
};
\addplot[forget plot, color=white!15!black] table[row sep=crcr] {%
0	0\\
100	0\\
};
\addlegendentry{\eqref{eq:newrelax}}

\end{axis}

\begin{axis}[%
width=4.521in,
height=0.944in,
at={(0.758in,0.481in)},
scale only axis,
bar shift auto,
xmin=0,
xmax=100,
ymin=0,
ymax=50,
axis background/.style={fill=white},
legend style={legend cell align=left, align=left, draw=white!15!black}
]
\addplot[ybar, bar width=3, fill=mycolor3, draw=black, area legend] table[row sep=crcr] {%
1	0\\
2	0\\
3	0\\
4	0\\
5	0\\
6	0\\
7	0\\
8	0\\
9	0\\
10	0\\
11	0\\
12	1\\
13	1\\
14	1\\
15	2\\
16	8\\
17	9\\
18	4\\
19	8\\
20	8\\
21	5\\
22	0\\
23	2\\
24	0\\
25	1\\
26	0\\
27	0\\
28	0\\
29	0\\
30	0\\
31	0\\
32	0\\
33	0\\
34	0\\
35	0\\
36	0\\
37	0\\
38	0\\
39	0\\
40	0\\
41	0\\
42	0\\
43	0\\
44	0\\
45	0\\
46	0\\
47	0\\
48	0\\
49	0\\
50	0\\
51	0\\
52	0\\
53	0\\
54	0\\
55	0\\
56	0\\
57	0\\
58	0\\
59	0\\
60	0\\
61	0\\
62	0\\
63	0\\
64	0\\
65	0\\
66	0\\
67	0\\
68	0\\
69	0\\
70	0\\
71	0\\
72	0\\
73	0\\
74	0\\
75	0\\
76	0\\
77	0\\
78	0\\
79	0\\
80	0\\
81	0\\
82	0\\
83	0\\
84	0\\
85	0\\
86	0\\
87	0\\
88	0\\
89	0\\
90	0\\
};
\addplot[forget plot, color=white!15!black] table[row sep=crcr] {%
0	0\\
100	0\\
};
\addlegendentry{\eqref{eq:1normrelax}}

\end{axis}
\end{tikzpicture}%
}
	\end{center}
	\caption{Left: Noise level $\|\epsilon\|$ versus distance $\|\x-\x_S\|$ between the obtained solution $\x$ and the oracle solution $\x_S$ for the three methods \eqref{eq:myrelax},\eqref{eq:newrelax} and \eqref{eq:1normrelax}. Right, histogram of the number of non-zero elements in the obtained solutions.}
	\label{fig:hist}
\end{figure}

\subsection{Point-set Registration with Outliers}

Next we consider registration of 2D point clouds. 
We assume that we have a set of model points $\{\p_i\}_{i=1}^N$ that should be registered to $\{\q_i\}_{i=1}^N$ by minimizing $\sum_{i=1}^N\left\|sR \p_i+\t-\q_i\right\|^2$  . Here $sR$ is a scaled rotation of the form $\left(\begin{matrix}
a & -b \\
b & a
\end{matrix}\right)$ and $t\in \R^2$ is a translation vector. Since the residuals are linear in the parameters $a,b,\t$, we can by column-stacking them write the problem as $\|M\y-\v \|^2$, where the vector $\y$ contains the unknowns $a,b,\t$. We further assume that the point matches contain outliers that needs to be removed. Therefore we add a sparse vector $\x$ whose non-zero entries allows the solution to have large errors. We thus want to solve
\begin{equation}
\min_{\x,\y}\mu \|\x\|_0 + \|M\y-\v+\x\|^2.
\label{eq:orgobjreg}
\end{equation}
The minimization over $\y$ can be carried out in closed form by noting that $\y = (M^T M)^{-1}M^T(\v-\x)$. Inserting into \eqref{eq:orgobjreg} which gives the objective function \eqref{eq:sparsity}, where $A = I - M(M^TM)^{-1}M^T$ and 
$\b = A\v$. The matrix $A$ is a projection onto the complement of the column space of $M$, and therefore has a 4 dimensional null space. 

Figure~\ref{fig:syntreg} shows the results of a synthetic experiment with $500$ problem instances. 
The data was generated by first selecting $100$ random Gaussian 2D points. 
We then divided these into two groups of $60$ and $40$ respectively and transformed these using two different random similarity transformations. This way the data supports two strong hypotheses which yields a problem which is much more difficult than what adding random uniformly distributed outliers does. The transformations were generated by taking $a$ and $b$ to be Gaussian with mean $0$ and variance $1$, and selecting $\t$ to be 2D-Gaussian with mean $(0,0)$ and covariance $5I$. 
We compare the three relaxations \eqref{eq:1normrelax} with $\lambda = 2$, \eqref{eq:myrelax} with $\mu=1$ and \eqref{eq:newrelax} with $\mu=1$ and $\lambda=0.5$. (The reason for using $\lambda = 2$ in \eqref{eq:1normrelax} and $\mu=1$ in \eqref{eq:myrelax} is that this gives the same threshold in the corresponding proximal operators.)

All methods where initialized with the least squares solution $\min_x\|A \x - \b\|^2$.
In the left histogram of Figure~\ref{fig:syntreg} we plot the data fit with respect to the inlier residuals (corresponding to the first 60 points, that supports the larger hypothesis). In the right one we plot the number of residuals determined to be outliers.
When starting from the least squares initialization the formulation \eqref{eq:myrelax} frequently gets stuck in solutions with poor data fit that are dense and close to the least squares solution. However when it converges to the correct solution it gives a much better data fit then the $\ell^1$ norm formulation \eqref{eq:1normrelax} due to its lack of bias. The added $\ell^1$ term to helps  \eqref{eq:newrelax} converge to the correct solution with a good data fit. Note that the number of outliers are in many cases smaller than $40$ due to the randomness of the data.
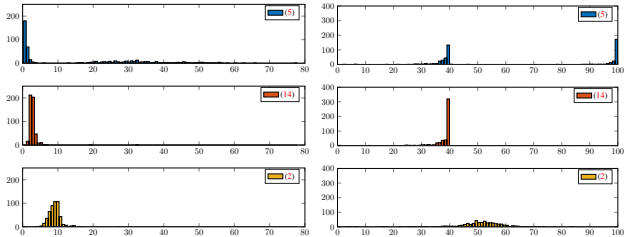
\begin{figure}[htb]
	\begin{center}
		\resizebox{41mm}{!}{
%
%
\definecolor{mycolor1}{rgb}{0.00000,0.44700,0.74100}%
\definecolor{mycolor2}{rgb}{0.85000,0.32500,0.09800}%
\definecolor{mycolor3}{rgb}{0.92900,0.69400,0.12500}%
\begin{tikzpicture}
\def\barwidth{3}
\begin{axis}[%
width=4.521in,
height=0.944in,
at={(0.758in,3.103in)},
scale only axis,
bar shift auto,
xmin=0,
xmax=80,
ymin=0,
ymax=250,
axis background/.style={fill=white},
legend style={legend cell align=left, align=left, draw=white!15!black}
]
\addplot[ybar, bar width=\barwidth, fill=mycolor1, draw=black, area legend] table[row sep=crcr] {%
0.651452948866196	180\\
1.42689805335731	69\\
2.20234315784842	15\\
2.97778826233953	5\\
3.75323336683064	2\\
4.52867847132175	0\\
5.30412357581286	0\\
6.07956868030397	1\\
6.85501378479508	0\\
7.63045888928619	0\\
8.4059039937773	0\\
9.18134909826841	0\\
9.95679420275952	0\\
10.7322393072506	0\\
11.5076844117417	0\\
12.2831295162329	0\\
13.058574620724	1\\
13.8340197252151	0\\
14.6094648297062	0\\
15.3849099341973	1\\
16.1603550386884	2\\
16.9358001431795	3\\
17.7112452476706	0\\
18.4866903521617	3\\
19.2621354566528	4\\
20.037580561144	5\\
20.8130256656351	8\\
21.5884707701262	1\\
22.3639158746173	5\\
23.1393609791084	6\\
23.9148060835995	5\\
24.6902511880906	8\\
25.4656962925817	4\\
26.2411413970728	10\\
27.016586501564	6\\
27.7920316060551	4\\
28.5674767105462	5\\
29.3429218150373	9\\
30.1183669195284	5\\
30.8938120240195	10\\
31.6692571285106	8\\
32.4447022330017	13\\
33.2201473374928	5\\
33.9955924419839	5\\
34.7710375464751	7\\
35.5464826509662	7\\
36.3219277554573	3\\
37.0973728599484	3\\
37.8728179644395	7\\
38.6482630689306	3\\
39.4237081734217	2\\
40.1991532779128	3\\
40.9745983824039	3\\
41.7500434868951	3\\
42.5254885913862	1\\
43.3009336958773	2\\
44.0763788003684	4\\
44.8518239048595	3\\
45.6272690093506	6\\
46.4027141138417	4\\
47.1781592183328	3\\
47.9536043228239	1\\
48.729049427315	1\\
49.5044945318062	3\\
50.2799396362973	4\\
51.0553847407884	4\\
51.8308298452795	1\\
52.6062749497706	1\\
53.3817200542617	1\\
54.1571651587528	1\\
54.9326102632439	1\\
55.708055367735	4\\
56.4835004722261	1\\
57.2589455767173	0\\
58.0343906812084	1\\
58.8098357856995	0\\
59.5852808901906	0\\
60.3607259946817	1\\
61.1361710991728	0\\
61.9116162036639	0\\
62.687061308155	1\\
63.4625064126461	0\\
64.2379515171372	0\\
65.0133966216284	0\\
65.7888417261195	1\\
66.5642868306106	0\\
67.3397319351017	0\\
68.1151770395928	0\\
68.8906221440839	0\\
69.666067248575	0\\
70.4415123530661	0\\
71.2169574575572	0\\
71.9924025620484	0\\
72.7678476665395	0\\
73.5432927710306	0\\
74.3187378755217	0\\
75.0941829800128	0\\
75.8696280845039	0\\
76.645073188995	0\\
77.4205182934861	1\\
};
\addplot[forget plot, color=white!15!black] table[row sep=crcr] {%
0	0\\
80	0\\
};
\addlegendentry{\eqref{eq:myrelax}}

\end{axis}

\begin{axis}[%
width=4.521in,
height=0.944in,
at={(0.758in,1.792in)},
scale only axis,
bar shift auto,
xmin=0,
xmax=80,
ymin=0,
ymax=250,
axis background/.style={fill=white},
legend style={legend cell align=left, align=left, draw=white!15!black}
]
\addplot[ybar, bar width=\barwidth, fill=mycolor2, draw=black, area legend] table[row sep=crcr] {%
0.651452948866196	0\\
1.42689805335731	16\\
2.20234315784842	212\\
2.97778826233953	203\\
3.75323336683064	47\\
4.52867847132175	8\\
5.30412357581286	10\\
6.07956868030397	1\\
6.85501378479508	1\\
7.63045888928619	0\\
8.4059039937773	0\\
9.18134909826841	0\\
9.95679420275952	0\\
10.7322393072506	0\\
11.5076844117417	0\\
12.2831295162329	0\\
13.058574620724	0\\
13.8340197252151	0\\
14.6094648297062	0\\
15.3849099341973	0\\
16.1603550386884	0\\
16.9358001431795	0\\
17.7112452476706	0\\
18.4866903521617	0\\
19.2621354566528	1\\
20.037580561144	0\\
20.8130256656351	0\\
21.5884707701262	0\\
22.3639158746173	0\\
23.1393609791084	0\\
23.9148060835995	0\\
24.6902511880906	0\\
25.4656962925817	0\\
26.2411413970728	0\\
27.016586501564	0\\
27.7920316060551	0\\
28.5674767105462	0\\
29.3429218150373	0\\
30.1183669195284	0\\
30.8938120240195	0\\
31.6692571285106	0\\
32.4447022330017	1\\
33.2201473374928	0\\
33.9955924419839	0\\
34.7710375464751	0\\
35.5464826509662	0\\
36.3219277554573	0\\
37.0973728599484	0\\
37.8728179644395	0\\
38.6482630689306	0\\
39.4237081734217	0\\
40.1991532779128	0\\
40.9745983824039	0\\
41.7500434868951	0\\
42.5254885913862	0\\
43.3009336958773	0\\
44.0763788003684	0\\
44.8518239048595	0\\
45.6272690093506	0\\
46.4027141138417	0\\
47.1781592183328	0\\
47.9536043228239	0\\
48.729049427315	0\\
49.5044945318062	0\\
50.2799396362973	0\\
51.0553847407884	0\\
51.8308298452795	0\\
52.6062749497706	0\\
53.3817200542617	0\\
54.1571651587528	0\\
54.9326102632439	0\\
55.708055367735	0\\
56.4835004722261	0\\
57.2589455767173	0\\
58.0343906812084	0\\
58.8098357856995	0\\
59.5852808901906	0\\
60.3607259946817	0\\
61.1361710991728	0\\
61.9116162036639	0\\
62.687061308155	0\\
63.4625064126461	0\\
64.2379515171372	0\\
65.0133966216284	0\\
65.7888417261195	0\\
66.5642868306106	0\\
67.3397319351017	0\\
68.1151770395928	0\\
68.8906221440839	0\\
69.666067248575	0\\
70.4415123530661	0\\
71.2169574575572	0\\
71.9924025620484	0\\
72.7678476665395	0\\
73.5432927710306	0\\
74.3187378755217	0\\
75.0941829800128	0\\
75.8696280845039	0\\
76.645073188995	0\\
77.4205182934861	0\\
};
\addplot[forget plot, color=white!15!black] table[row sep=crcr] {%
0	0\\
80	0\\
};
\addlegendentry{\eqref{eq:newrelax}}

\end{axis}

\begin{axis}[%
width=4.521in,
height=0.944in,
at={(0.758in,0.481in)},
scale only axis,
bar shift auto,
xmin=0,
xmax=80,
ymin=0,
ymax=250,
axis background/.style={fill=white},
legend style={legend cell align=left, align=left, draw=white!15!black}
]
\addplot[ybar, bar width=\barwidth, fill=mycolor3, draw=black, area legend] table[row sep=crcr] {%
0.651452948866196	0\\
1.42689805335731	0\\
2.20234315784842	0\\
2.97778826233953	0\\
3.75323336683064	1\\
4.52867847132175	1\\
5.30412357581286	4\\
6.07956868030397	15\\
6.85501378479508	36\\
7.63045888928619	65\\
8.4059039937773	90\\
9.18134909826841	107\\
9.95679420275952	107\\
10.7322393072506	44\\
11.5076844117417	10\\
12.2831295162329	7\\
13.058574620724	2\\
13.8340197252151	4\\
14.6094648297062	5\\
15.3849099341973	0\\
16.1603550386884	1\\
16.9358001431795	0\\
17.7112452476706	0\\
18.4866903521617	0\\
19.2621354566528	0\\
20.037580561144	0\\
20.8130256656351	0\\
21.5884707701262	0\\
22.3639158746173	0\\
23.1393609791084	0\\
23.9148060835995	0\\
24.6902511880906	0\\
25.4656962925817	0\\
26.2411413970728	0\\
27.016586501564	1\\
27.7920316060551	0\\
28.5674767105462	0\\
29.3429218150373	0\\
30.1183669195284	0\\
30.8938120240195	0\\
31.6692571285106	0\\
32.4447022330017	0\\
33.2201473374928	0\\
33.9955924419839	0\\
34.7710375464751	0\\
35.5464826509662	0\\
36.3219277554573	0\\
37.0973728599484	0\\
37.8728179644395	0\\
38.6482630689306	0\\
39.4237081734217	0\\
40.1991532779128	0\\
40.9745983824039	0\\
41.7500434868951	0\\
42.5254885913862	0\\
43.3009336958773	0\\
44.0763788003684	0\\
44.8518239048595	0\\
45.6272690093506	0\\
46.4027141138417	0\\
47.1781592183328	0\\
47.9536043228239	0\\
48.729049427315	0\\
49.5044945318062	0\\
50.2799396362973	0\\
51.0553847407884	0\\
51.8308298452795	0\\
52.6062749497706	0\\
53.3817200542617	0\\
54.1571651587528	0\\
54.9326102632439	0\\
55.708055367735	0\\
56.4835004722261	0\\
57.2589455767173	0\\
58.0343906812084	0\\
58.8098357856995	0\\
59.5852808901906	0\\
60.3607259946817	0\\
61.1361710991728	0\\
61.9116162036639	0\\
62.687061308155	0\\
63.4625064126461	0\\
64.2379515171372	0\\
65.0133966216284	0\\
65.7888417261195	0\\
66.5642868306106	0\\
67.3397319351017	0\\
68.1151770395928	0\\
68.8906221440839	0\\
69.666067248575	0\\
70.4415123530661	0\\
71.2169574575572	0\\
71.9924025620484	0\\
72.7678476665395	0\\
73.5432927710306	0\\
74.3187378755217	0\\
75.0941829800128	0\\
75.8696280845039	0\\
76.645073188995	0\\
77.4205182934861	0\\
};
\addplot[forget plot, color=white!15!black] table[row sep=crcr] {%
0	0\\
80	0\\
};
\addlegendentry{\eqref{eq:1normrelax}}

\end{axis}
\end{tikzpicture}%
}
		\resizebox{41mm}{!}{
%
%
\definecolor{mycolor1}{rgb}{0.00000,0.44700,0.74100}%
\definecolor{mycolor2}{rgb}{0.85000,0.32500,0.09800}%
\definecolor{mycolor3}{rgb}{0.92900,0.69400,0.12500}%
\begin{tikzpicture}
\def\barwidth{3}
\begin{axis}[%
width=4.521in,
height=0.944in,
at={(0.758in,3.103in)},
scale only axis,
bar shift auto,
xmin=0,
xmax=100,
ymin=0,
ymax=400,
axis background/.style={fill=white},
legend style={legend cell align=left, align=left, draw=white!15!black}
]
\addplot[ybar, bar width=\barwidth, fill=mycolor1, draw=black, area legend] table[row sep=crcr] {%
0.5	0\\
1.5	0\\
2.5	1\\
3.5	0\\
4.5	0\\
5.5	0\\
6.5	2\\
7.5	0\\
8.5	0\\
9.5	0\\
10.5	0\\
11.5	0\\
12.5	0\\
13.5	0\\
14.5	0\\
15.5	0\\
16.5	0\\
17.5	1\\
18.5	0\\
19.5	0\\
20.5	0\\
21.5	1\\
22.5	0\\
23.5	0\\
24.5	0\\
25.5	1\\
26.5	0\\
27.5	3\\
28.5	3\\
29.5	2\\
30.5	3\\
31.5	6\\
32.5	2\\
33.5	4\\
34.5	5\\
35.5	6\\
36.5	23\\
37.5	29\\
38.5	44\\
39.5	133\\
40.5	2\\
41.5	0\\
42.5	0\\
43.5	0\\
44.5	0\\
45.5	0\\
46.5	1\\
47.5	0\\
48.5	0\\
49.5	0\\
50.5	0\\
51.5	0\\
52.5	0\\
53.5	0\\
54.5	0\\
55.5	0\\
56.5	0\\
57.5	0\\
58.5	0\\
59.5	0\\
60.5	0\\
61.5	0\\
62.5	0\\
63.5	0\\
64.5	0\\
65.5	0\\
66.5	0\\
67.5	0\\
68.5	0\\
69.5	0\\
70.5	0\\
71.5	0\\
72.5	0\\
73.5	0\\
74.5	0\\
75.5	0\\
76.5	0\\
77.5	0\\
78.5	1\\
79.5	0\\
80.5	0\\
81.5	0\\
82.5	0\\
83.5	0\\
84.5	0\\
85.5	0\\
86.5	0\\
87.5	1\\
88.5	1\\
89.5	1\\
90.5	2\\
91.5	1\\
92.5	1\\
93.5	2\\
94.5	1\\
95.5	2\\
96.5	7\\
97.5	13\\
98.5	23\\
99.5	172\\
};
\addplot[forget plot, color=white!15!black] table[row sep=crcr] {%
0	0\\
100	0\\
};
\addlegendentry{\eqref{eq:myrelax}}

\end{axis}

\begin{axis}[%
width=4.521in,
height=0.944in,
at={(0.758in,1.792in)},
scale only axis,
bar shift auto,
xmin=0,
xmax=100,
ymin=0,
ymax=400,
axis background/.style={fill=white},
legend style={legend cell align=left, align=left, draw=white!15!black}
]
\addplot[ybar, bar width=\barwidth, fill=mycolor2, draw=black, area legend] table[row sep=crcr] {%
0.5	3\\
1.5	0\\
2.5	0\\
3.5	0\\
4.5	0\\
5.5	1\\
6.5	0\\
7.5	1\\
8.5	0\\
9.5	0\\
10.5	0\\
11.5	2\\
12.5	1\\
13.5	0\\
14.5	0\\
15.5	1\\
16.5	0\\
17.5	0\\
18.5	1\\
19.5	1\\
20.5	1\\
21.5	1\\
22.5	1\\
23.5	1\\
24.5	4\\
25.5	3\\
26.5	2\\
27.5	3\\
28.5	0\\
29.5	4\\
30.5	6\\
31.5	5\\
32.5	8\\
33.5	4\\
34.5	6\\
35.5	19\\
36.5	21\\
37.5	36\\
38.5	39\\
39.5	320\\
40.5	1\\
41.5	1\\
42.5	1\\
43.5	0\\
44.5	0\\
45.5	0\\
46.5	0\\
47.5	0\\
48.5	0\\
49.5	0\\
50.5	0\\
51.5	0\\
52.5	0\\
53.5	0\\
54.5	0\\
55.5	0\\
56.5	0\\
57.5	0\\
58.5	0\\
59.5	0\\
60.5	0\\
61.5	0\\
62.5	0\\
63.5	0\\
64.5	0\\
65.5	0\\
66.5	0\\
67.5	0\\
68.5	0\\
69.5	0\\
70.5	0\\
71.5	0\\
72.5	1\\
73.5	0\\
74.5	0\\
75.5	0\\
76.5	0\\
77.5	0\\
78.5	0\\
79.5	0\\
80.5	0\\
81.5	0\\
82.5	0\\
83.5	0\\
84.5	0\\
85.5	0\\
86.5	0\\
87.5	0\\
88.5	0\\
89.5	0\\
90.5	0\\
91.5	0\\
92.5	0\\
93.5	0\\
94.5	0\\
95.5	0\\
96.5	0\\
97.5	0\\
98.5	0\\
99.5	1\\
};
\addplot[forget plot, color=white!15!black] table[row sep=crcr] {%
0	0\\
100	0\\
};
\addlegendentry{\eqref{eq:newrelax}}

\end{axis}

\begin{axis}[%
width=4.521in,
height=0.944in,
at={(0.758in,0.481in)},
scale only axis,
bar shift auto,
xmin=0,
xmax=100,
ymin=0,
ymax=400,
axis background/.style={fill=white},
legend style={legend cell align=left, align=left, draw=white!15!black}
]
\addplot[ybar, bar width=\barwidth, fill=mycolor3, draw=black, area legend] table[row sep=crcr] {%
0.5	0\\
1.5	0\\
2.5	0\\
3.5	1\\
4.5	1\\
5.5	0\\
6.5	1\\
7.5	0\\
8.5	0\\
9.5	0\\
10.5	0\\
11.5	0\\
12.5	0\\
13.5	0\\
14.5	0\\
15.5	0\\
16.5	1\\
17.5	2\\
18.5	0\\
19.5	0\\
20.5	0\\
21.5	0\\
22.5	1\\
23.5	0\\
24.5	0\\
25.5	1\\
26.5	2\\
27.5	2\\
28.5	0\\
29.5	2\\
30.5	3\\
31.5	1\\
32.5	0\\
33.5	0\\
34.5	1\\
35.5	3\\
36.5	2\\
37.5	5\\
38.5	6\\
39.5	5\\
40.5	5\\
41.5	5\\
42.5	4\\
43.5	9\\
44.5	12\\
45.5	17\\
46.5	24\\
47.5	16\\
48.5	23\\
49.5	43\\
50.5	30\\
51.5	30\\
52.5	39\\
53.5	27\\
54.5	28\\
55.5	27\\
56.5	22\\
57.5	19\\
58.5	19\\
59.5	12\\
60.5	11\\
61.5	3\\
62.5	8\\
63.5	6\\
64.5	5\\
65.5	1\\
66.5	2\\
67.5	3\\
68.5	1\\
69.5	2\\
70.5	2\\
71.5	2\\
72.5	0\\
73.5	1\\
74.5	0\\
75.5	0\\
76.5	0\\
77.5	1\\
78.5	0\\
79.5	0\\
80.5	0\\
81.5	0\\
82.5	0\\
83.5	1\\
84.5	0\\
85.5	0\\
86.5	0\\
87.5	0\\
88.5	0\\
89.5	0\\
90.5	0\\
91.5	0\\
92.5	0\\
93.5	0\\
94.5	0\\
95.5	0\\
96.5	0\\
97.5	0\\
98.5	0\\
99.5	0\\
};
\addplot[forget plot, color=white!15!black] table[row sep=crcr] {%
0	0\\
100	0\\
};
\addlegendentry{\eqref{eq:1normrelax}}

\end{axis}
\end{tikzpicture}%
}
	\end{center}
	\caption{Results form the synthetic registration experiment. Left: Data fit of the resulting estimation to the true lnliers. Right: Number of estimated outliers.}
	\label{fig:syntreg}
\end{figure}

We also include a few problem instances with real data. Here we matched SIFT descriptors between two images, as shown in Figure~\ref{fig:realreg}, to generate the two point sets 
$\{\p_i\}_{i=1}^N$ and $\{\q_i\}_{i=1}^N$. We then registered the points sets using the formulations
\eqref{eq:myrelax} with $\mu = 20^2$ and \eqref{eq:1normrelax} with $\lambda = 10$ 
(which in both cases corresponds to a 20 pixel outlier threshold in a $3072 \times 2048$ image).
For \eqref{eq:newrelax} we used $\mu = 20^2$ and $\lambda = 5$.
\begin{figure}[htb]
	\def\w{60mm}
	\begin{center}
		\includegraphics[width=\w]{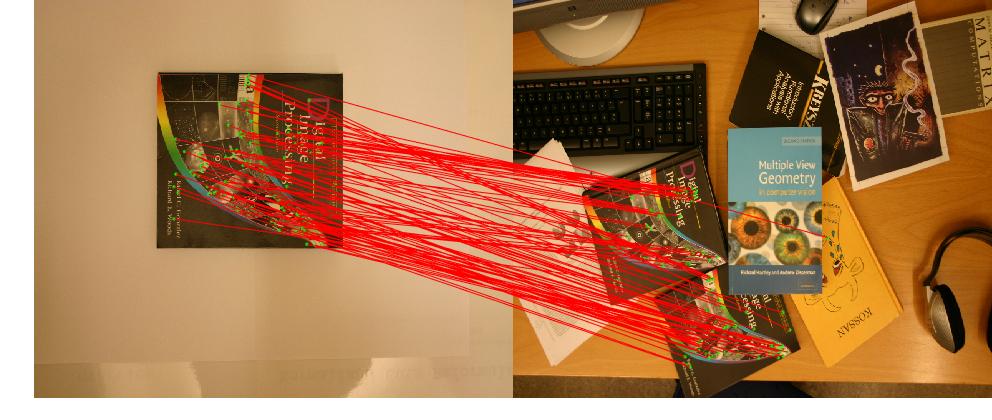}
	\end{center}
	\caption{Matches between two of the images used in Figure~\ref{fig:realregresult}.}
	\label{fig:realreg}
\end{figure}

The results are shown in Figure~\ref{fig:realreg}. In the first problem instance (first row) we used an image which generates one strong hypothesis. Here both \eqref{eq:newrelax} and \eqref{eq:1normrelax} produce good results. In contrast \eqref{eq:myrelax} immediately gets stuck in the least squares solution for which all residuals are above the threshold. In the second instance there are two strong hypotheses. The incorrect one introduces a systematic bias that effects \eqref{eq:1normrelax} more than \eqref{eq:newrelax}. As a result the registration obtained by \eqref{eq:newrelax} is better than that of \eqref{eq:myrelax} and the number of determined inliers is larger.
\begin{figure}[htb]
	\def\w{27mm}
	\begin{center}
		\includegraphics[width=\w]{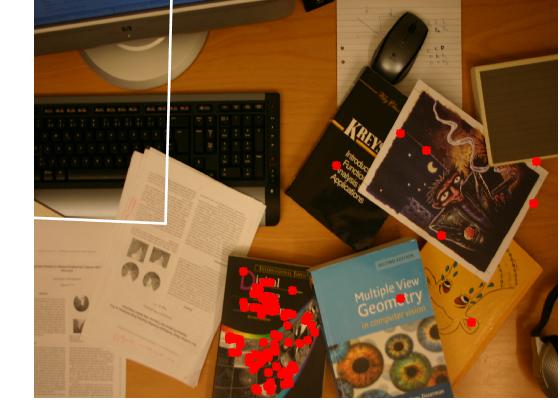}
		\includegraphics[width=\w]{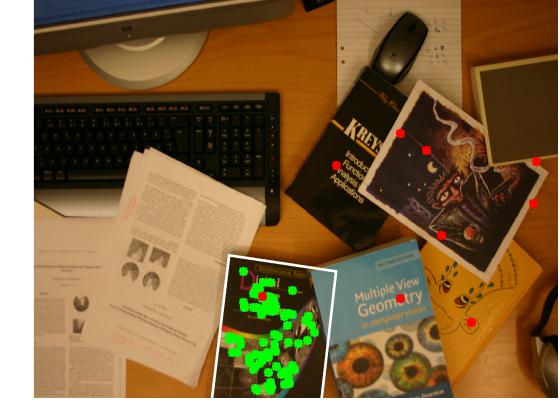}
		\includegraphics[width=\w]{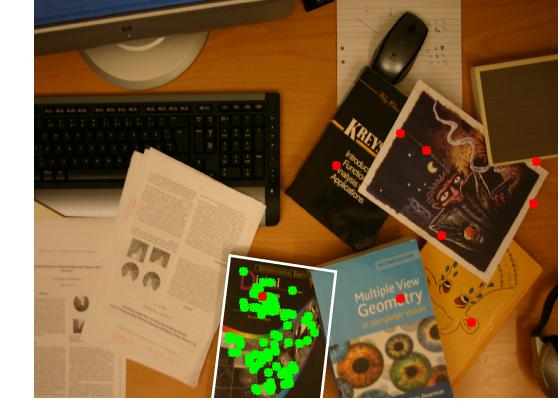}
		\includegraphics[width=\w]{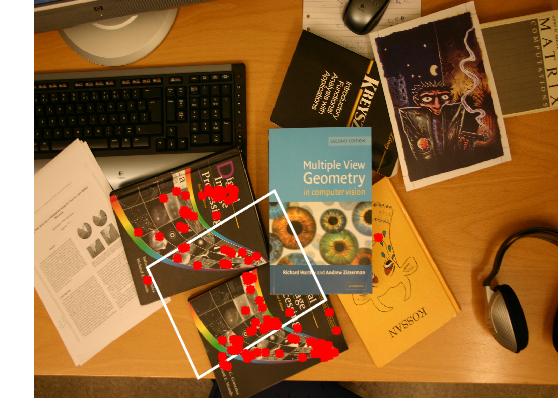}
		\includegraphics[width=\w]{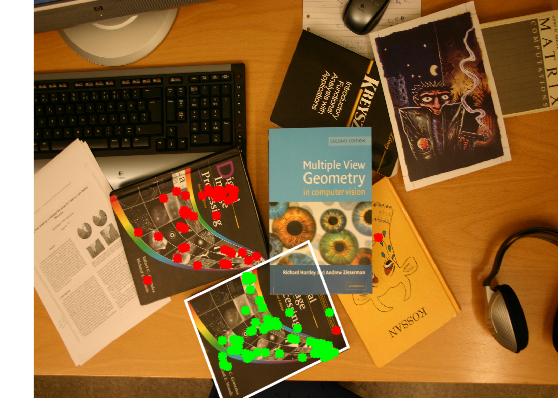}
		\includegraphics[width=\w]{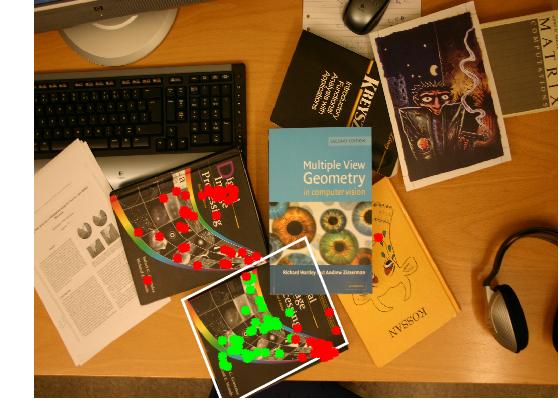}
	\end{center}
	\caption{Results form the two real registration experiment. From left to right: \eqref{eq:myrelax}, \eqref{eq:newrelax}, \eqref{eq:1normrelax}.
		Red means that the point was classified as outlier, green inlier. White frame shows registration of the model book under the estimated transformation. }
	\label{fig:realregresult}
\end{figure}

\subsection{Non-rigid Structure from Motion}\label{sec:nonrigid}
\begin{figure*}
	\begin{center}
		\def\w{9mm}
		\begin{tabular}{|c|c|c|c|}
			\hline & & &\\
			\includegraphics[width=\w]{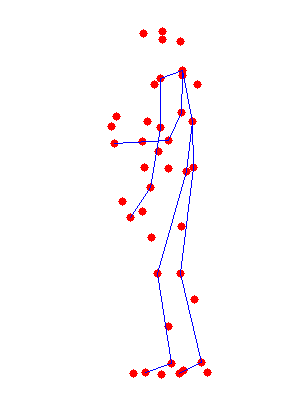}
			\includegraphics[width=\w]{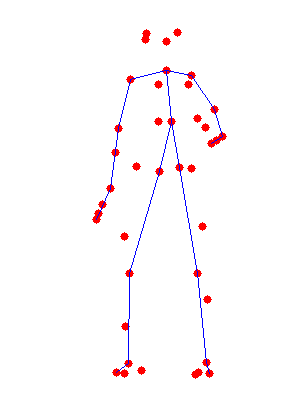}
			\includegraphics[width=\w]{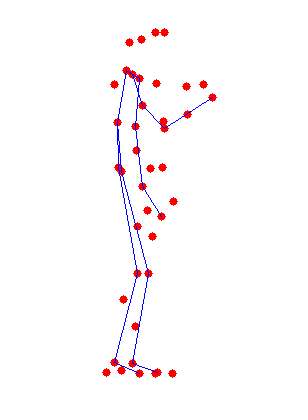}
			\includegraphics[width=\w]{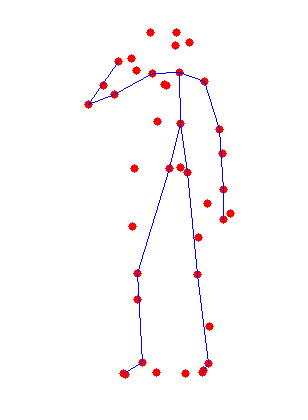}&
			\includegraphics[width=\w]{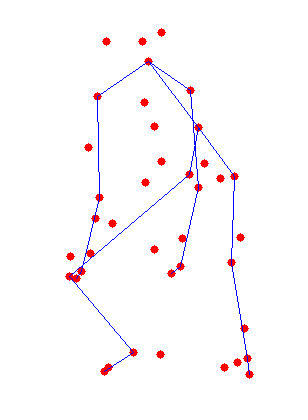}
			\includegraphics[width=\w]{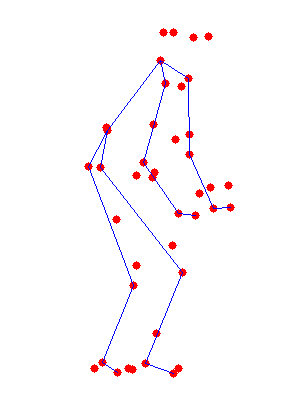}
			\includegraphics[width=\w]{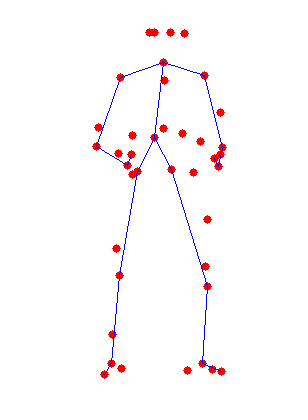}
			\includegraphics[width=\w]{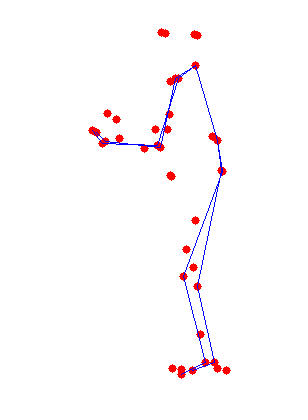}&
			\includegraphics[width=\w]{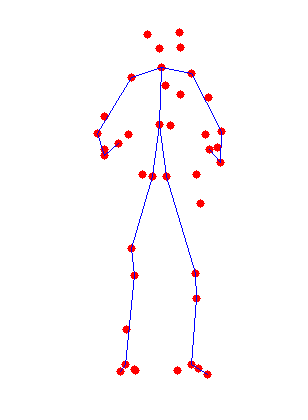}
			\includegraphics[width=\w]{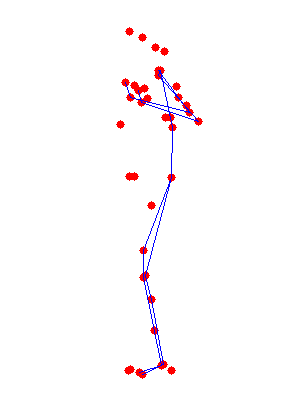}
			\includegraphics[width=\w]{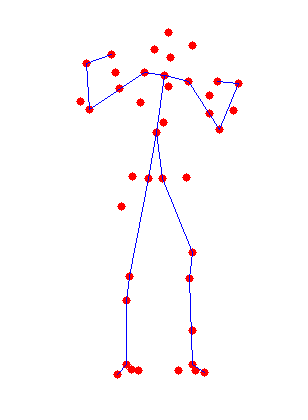}
			\includegraphics[width=\w]{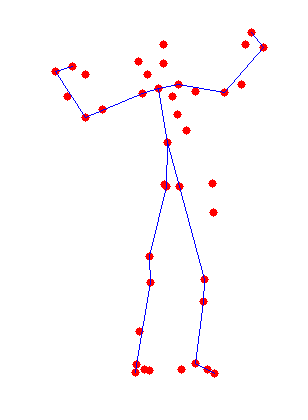}&
			\includegraphics[width=\w]{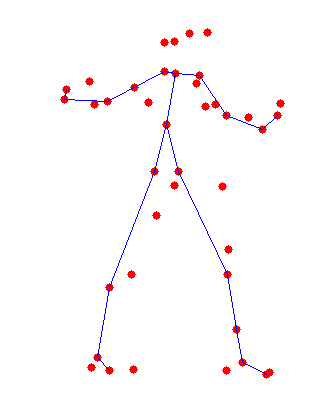}
			\includegraphics[width=\w]{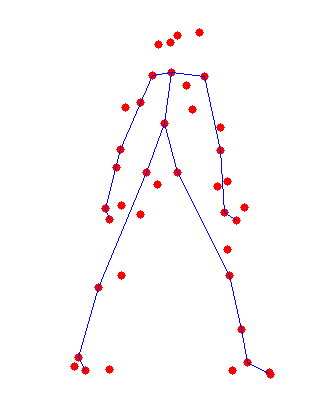}
			\includegraphics[width=\w]{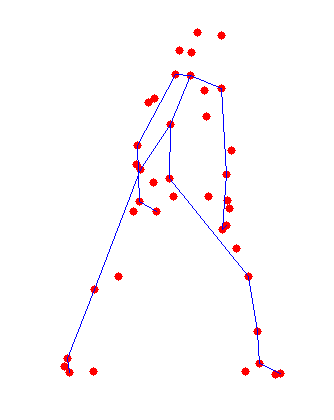}
			\includegraphics[width=\w]{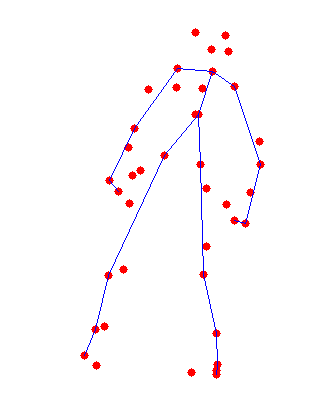}\\
			\emph{Drink} & \emph{Pick-up} &
			\emph{Stretch} & \emph{Yoga} \\
			\hline
		\end{tabular}
	\end{center}
	\caption{Four images from each of the MOCAP data sets.}
	\label{fig:mocapshapes}
	\begin{center}
		\def\w{39mm}
		\resizebox{\w}{!}{
%
%
\begin{tikzpicture}

\begin{axis}[%
width=4.52083333333333in,
height=3.565625in,
scale only axis,
xmin=10,
xmax=100,
ymin=0,
ymax=120,
legend style={draw=black,fill=white,legend cell align=left}
]
\addplot [
color=blue,
solid,
line width=2.0pt
]
table[row sep=crcr]{
10 2.2487218381549\\
11.8367346938776 2.35096121475605\\
13.6734693877551 2.26860270673283\\
15.5102040816327 2.31770680819378\\
17.3469387755102 3.77212921042247\\
19.1836734693878 3.27181416110025\\
21.0204081632653 4.97858207697487\\
22.8571428571429 11.20367232543\\
24.6938775510204 10.4301608854468\\
26.530612244898 10.5847667514785\\
28.3673469387755 10.390808097022\\
30.2040816326531 10.4331855878288\\
32.0408163265306 10.372510131131\\
33.8775510204082 10.2420236697738\\
35.7142857142857 10.3415130583021\\
37.5510204081633 10.3031748604013\\
39.3877551020408 11.7990471364654\\
41.2244897959184 17.0231023182291\\
43.0612244897959 17.0231023182291\\
44.8979591836735 49.9309849275814\\
46.734693877551 49.9309849275814\\
48.5714285714286 49.9309849275814\\
50.4081632653061 49.9309849275814\\
52.2448979591837 49.9309849275814\\
54.0816326530612 49.9309849275814\\
55.9183673469388 49.9309849275814\\
57.7551020408163 49.9309849275814\\
59.5918367346939 49.9309849275814\\
61.4285714285714 49.9309849275814\\
63.265306122449 49.9309849275814\\
65.1020408163265 49.9309849275814\\
66.9387755102041 49.9309849275814\\
68.7755102040816 49.9309849275814\\
70.6122448979592 49.9309849275814\\
72.4489795918367 49.9309849275814\\
74.2857142857143 49.9309849275814\\
76.1224489795918 49.9309849275814\\
77.9591836734694 49.9309849275814\\
79.7959183673469 49.9309849275814\\
81.6326530612245 49.9309849275814\\
83.469387755102 49.9309849275814\\
85.3061224489796 49.9309849275814\\
87.1428571428571 49.9309849275814\\
88.9795918367347 49.9309849275814\\
90.8163265306122 49.9309849275814\\
92.6530612244898 49.9309849275814\\
94.4897959183673 49.9309849275814\\
96.3265306122449 49.9309849275814\\
98.1632653061225 49.9309849275814\\
100 49.9309849275814\\
};
\addlegendentry{\eqref{eq:Srelax}};

\addplot [
color=green!50!black,
solid,
line width=2.0pt
]
table[row sep=crcr]{
10 12.2685108896105\\
11.8367346938776 17.4788856172772\\
13.6734693877551 17.4788856172772\\
15.5102040816327 17.4788856172772\\
17.3469387755102 17.4788856172772\\
19.1836734693878 17.4788856172772\\
21.0204081632653 17.4788856172772\\
22.8571428571429 17.4788856172773\\
24.6938775510204 17.4788856172773\\
26.530612244898 17.4788856172772\\
28.3673469387755 17.4788856172772\\
30.2040816326531 17.4788856172772\\
32.0408163265306 17.4788856172772\\
33.8775510204082 17.4788856172772\\
35.7142857142857 17.4788856172772\\
37.5510204081633 17.4788856172772\\
39.3877551020408 17.4788856172772\\
41.2244897959184 17.4788856172772\\
43.0612244897959 50.0012055065942\\
44.8979591836735 50.0012055065942\\
46.734693877551 50.0012055065942\\
48.5714285714286 50.0012055065942\\
50.4081632653061 50.0012055065942\\
52.2448979591837 50.0012055065942\\
54.0816326530612 50.0012055065942\\
55.9183673469388 50.0012055065942\\
57.7551020408163 50.0012055065942\\
59.5918367346939 50.0012055065942\\
61.4285714285714 50.0012055065942\\
63.265306122449 50.0012055065942\\
65.1020408163265 50.0012055065942\\
66.9387755102041 50.0012055065942\\
68.7755102040816 50.0012055065942\\
70.6122448979592 50.0012055065942\\
72.4489795918367 50.0012055065942\\
74.2857142857143 50.0012055065942\\
76.1224489795918 50.0012055065942\\
77.9591836734694 50.0012055065942\\
79.7959183673469 50.0012055065942\\
81.6326530612245 50.0012055065942\\
83.469387755102 50.0012055065942\\
85.3061224489796 50.0012055065942\\
87.1428571428571 50.0012055065942\\
88.9795918367347 50.0012055065942\\
90.8163265306122 50.0012055065942\\
92.6530612244898 50.0012055065942\\
94.4897959183673 50.0012055065942\\
96.3265306122449 50.0012055065942\\
98.1632653061225 50.0012055065942\\
100 50.0012055065942\\
};
\addlegendentry{\eqref{eq:crossover}};

\addplot [
color=red,
solid,
line width=2.0pt
]
table[row sep=crcr]{
10 22.3561799641621\\
11.8367346938776 24.9775553134997\\
13.6734693877551 27.0623587655142\\
15.5102040816327 29.2490275853574\\
17.3469387755102 31.5144399762911\\
19.1836734693878 33.841271633357\\
21.0204081632653 36.2164798369986\\
22.8571428571429 38.6301699052407\\
24.6938775510204 41.07476681512\\
26.530612244898 43.5444149865524\\
28.3673469387755 46.0345435922838\\
30.2040816326531 48.5415504447758\\
32.0408163265306 51.0625705624166\\
33.8775510204082 53.5953053463799\\
35.7142857142857 56.137895376005\\
37.5510204081633 58.6888248207992\\
39.3877551020408 61.2468489571058\\
41.2244897959184 63.8109387150185\\
43.0612244897959 66.3802378883346\\
44.8979591836735 68.6734271977859\\
46.734693877551 69.9999492075949\\
48.5714285714286 71.3522438790241\\
50.4081632653061 72.7287977548786\\
52.2448979591837 74.1281866224119\\
54.0816326530612 75.5490717893563\\
55.9183673469388 76.9901961343908\\
57.7551020408163 78.4503800335659\\
59.5918367346939 79.9285172442795\\
61.4285714285714 81.4235708112265\\
63.265306122449 82.9345690441588\\
65.1020408163265 84.4606016050976\\
66.9387755102041 86.0008157324295\\
68.7755102040816 87.5544126211267\\
70.6122448979592 89.1206439715389\\
72.4489795918367 90.6988087138813\\
74.2857142857143 92.2882499113406\\
76.1224489795918 93.8883518414067\\
77.9591836734694 95.4985372525707\\
79.7959183673469 97.1182647916821\\
81.6326530612245 98.7470265958947\\
83.469387755102 100.384346042243\\
85.3061224489796 102.029775647285\\
87.1428571428571 103.682895108907\\
88.9795918367347 105.343309482281\\
90.8163265306122 107.010647481988\\
92.6530612244898 108.684559902429\\
94.4897959183673 110.364718148917\\
96.3265306122449 112.050812872095\\
98.1632653061225 113.74255269868\\
100 115.43966305186\\
};
\addlegendentry{\eqref{eq:nuclearrelax}};

\end{axis}
\end{tikzpicture}%
}
		\resizebox{\w}{!}{
%
%
\begin{tikzpicture}

\begin{axis}[%
width=4.52083333333333in,
height=3.565625in,
scale only axis,
xmin=10,
xmax=100,
ymin=0,
ymax=120,
legend style={draw=black,fill=white,legend cell align=left}
]
\addplot [
color=blue,
solid,
line width=2.0pt
]
table[row sep=crcr]{
10 2.51586680935418\\
11.8367346938776 4.79570029262791\\
13.6734693877551 6.58219966488478\\
15.5102040816327 6.56982910549078\\
17.3469387755102 6.55502347049673\\
19.1836734693878 6.56283110928307\\
21.0204081632653 7.60222195621941\\
22.8571428571429 7.51821023551489\\
24.6938775510204 8.69739436168837\\
26.530612244898 8.69739436168838\\
28.3673469387755 25.3253540415262\\
30.2040816326531 25.3253540415262\\
32.0408163265306 25.3253540415261\\
33.8775510204082 25.3253540415262\\
35.7142857142857 46.0445806507741\\
37.5510204081633 46.0445806507741\\
39.3877551020408 46.0445806507741\\
41.2244897959184 46.0445806507741\\
43.0612244897959 46.0445806507741\\
44.8979591836735 46.0445806507741\\
46.734693877551 46.0445806507741\\
48.5714285714286 46.0445806507741\\
50.4081632653061 46.0445806507741\\
52.2448979591837 46.0445806507741\\
54.0816326530612 46.0445806507741\\
55.9183673469388 46.0445806507741\\
57.7551020408163 46.0445806507741\\
59.5918367346939 46.0445806507741\\
61.4285714285714 46.0445806507741\\
63.265306122449 46.0445806507741\\
65.1020408163265 46.0445806507741\\
66.9387755102041 46.0445806507741\\
68.7755102040816 46.0445806507741\\
70.6122448979592 46.0445806507741\\
72.4489795918367 46.0445806507741\\
74.2857142857143 46.0445806507741\\
76.1224489795918 46.0445806507741\\
77.9591836734694 46.0445806507741\\
79.7959183673469 46.0445806507741\\
81.6326530612245 46.0445806507741\\
83.469387755102 46.0445806507741\\
85.3061224489796 46.0445806507741\\
87.1428571428571 46.0445806507741\\
88.9795918367347 46.0445806507741\\
90.8163265306122 46.0445806507741\\
92.6530612244898 46.0445806507741\\
94.4897959183673 46.0445806507741\\
96.3265306122449 46.0445806507741\\
98.1632653061225 46.0445806507741\\
100 46.0445806507741\\
};
\addlegendentry{\eqref{eq:Srelax}};

\addplot [
color=green!50!black,
solid,
line width=2.0pt
]
table[row sep=crcr]{
10 10.3632649154152\\
11.8367346938776 10.3632649154153\\
13.6734693877551 10.3632649154153\\
15.5102040816327 10.3632649154153\\
17.3469387755102 10.3632649154152\\
19.1836734693878 25.6764053861315\\
21.0204081632653 25.6764053861315\\
22.8571428571429 25.6764053861315\\
24.6938775510204 25.6764053861316\\
26.530612244898 25.6764053861315\\
28.3673469387755 25.6764053861315\\
30.2040816326531 25.6764053861315\\
32.0408163265306 25.6764053861315\\
33.8775510204082 46.1275586638643\\
35.7142857142857 46.1275586638643\\
37.5510204081633 46.1275586638643\\
39.3877551020408 46.1275586638643\\
41.2244897959184 46.1275586638643\\
43.0612244897959 46.1275586638643\\
44.8979591836735 46.1275586638643\\
46.734693877551 46.1275586638643\\
48.5714285714286 46.1275586638643\\
50.4081632653061 46.1275586638643\\
52.2448979591837 46.1275586638643\\
54.0816326530612 46.1275586638643\\
55.9183673469388 46.1275586638643\\
57.7551020408163 46.1275586638643\\
59.5918367346939 46.1275586638643\\
61.4285714285714 46.1275586638643\\
63.265306122449 46.1275586638643\\
65.1020408163265 46.1275586638643\\
66.9387755102041 46.1275586638643\\
68.7755102040816 46.1275586638643\\
70.6122448979592 46.1275586638643\\
72.4489795918367 46.1275586638643\\
74.2857142857143 46.1275586638643\\
76.1224489795918 46.1275586638643\\
77.9591836734694 46.1275586638643\\
79.7959183673469 46.1275586638643\\
81.6326530612245 46.1275586638643\\
83.469387755102 46.1275586638643\\
85.3061224489796 46.1275586638643\\
87.1428571428571 46.1275586638643\\
88.9795918367347 46.1275586638643\\
90.8163265306122 46.1275586638643\\
92.6530612244898 46.1275586638643\\
94.4897959183673 46.1275586638643\\
96.3265306122449 46.1275586638643\\
98.1632653061225 46.1275586638643\\
100 46.1275586638643\\
};
\addlegendentry{\eqref{eq:crossover}};

\addplot [
color=red,
solid,
line width=2.0pt
]
table[row sep=crcr]{
10 22.462967394336\\
11.8367346938776 25.7316297831663\\
13.6734693877551 28.996007163981\\
15.5102040816327 32.2468702505701\\
17.3469387755102 35.4837267409966\\
19.1836734693878 38.7092726342738\\
21.0204081632653 41.5850624785141\\
22.8571428571429 43.7920016311301\\
24.6938775510204 46.0478600692426\\
26.530612244898 48.3450197944512\\
28.3673469387755 50.6772411829128\\
30.2040816326531 53.0393904730848\\
32.0408163265306 55.4272225678893\\
33.8775510204082 57.8372086325654\\
35.7142857142857 60.169788361565\\
37.5510204081633 61.4461915484647\\
39.3877551020408 62.7554855167297\\
41.2244897959184 64.0954362642709\\
43.0612244897959 65.4639554763518\\
44.8979591836735 66.8590948933999\\
46.734693877551 68.2790398841924\\
48.5714285714286 69.7221025198724\\
50.4081632653061 71.1867143852437\\
52.2448979591837 72.6714193126183\\
54.0816326530612 74.1748661800743\\
55.9183673469388 75.6958018798813\\
57.7551020408163 77.2330645332425\\
59.5918367346939 78.7855770037134\\
61.4285714285714 80.3523407427553\\
63.265306122449 81.93242998615\\
65.1020408163265 83.524986308677\\
66.9387755102041 85.1292135359811\\
68.7755102040816 86.7443730062863\\
70.6122448979592 88.3697791701841\\
72.4489795918367 90.0047955137004\\
74.2857142857143 91.6488307879144\\
76.1224489795918 93.3013355273214\\
77.9591836734694 94.9617988386922\\
79.7959183673469 96.6297454422198\\
81.6326530612245 98.3047329470908\\
83.469387755102 99.9863493442644\\
85.3061224489796 101.6742107\\
87.1428571428571 103.367959034541\\
88.9795918367347 105.067260371332\\
90.8163265306122 106.771802943036\\
92.6530612244898 108.481295541658\\
94.4897959183673 110.195466000915\\
96.3265306122449 111.91405979995\\
98.1632653061225 113.636838778331\\
100 115.363579953041\\
};
\addlegendentry{\eqref{eq:nuclearrelax}};

\end{axis}
\end{tikzpicture}%
}
		\resizebox{\w}{!}{
%
%
\begin{tikzpicture}

\begin{axis}[%
width=4.52083333333333in,
height=3.565625in,
scale only axis,
xmin=10,
xmax=100,
ymin=0,
ymax=120,
legend style={draw=black,fill=white,legend cell align=left}
]
\addplot [
color=blue,
solid,
line width=2.0pt
]
table[row sep=crcr]{
10 7.21002344060896\\
11.8367346938776 5.61573103776516\\
13.6734693877551 7.05016869250965\\
15.5102040816327 7.02996708299591\\
17.3469387755102 7.02964017810828\\
19.1836734693878 8.66031446598215\\
21.0204081632653 8.54284157596862\\
22.8571428571429 8.58222943521243\\
24.6938775510204 10.5920005649657\\
26.530612244898 21.8000191367291\\
28.3673469387755 21.8000191367291\\
30.2040816326531 21.8000191367291\\
32.0408163265306 21.8000191367291\\
33.8775510204082 21.8000191367291\\
35.7142857142857 21.8000191367291\\
37.5510204081633 21.8000191367291\\
39.3877551020408 21.8000191367291\\
41.2244897959184 21.800019136729\\
43.0612244897959 21.800019136729\\
44.8979591836735 21.8000191367291\\
46.734693877551 21.8000191367291\\
48.5714285714286 21.8000191367291\\
50.4081632653061 55.0723111996961\\
52.2448979591837 55.0723111996961\\
54.0816326530612 55.0723111996961\\
55.9183673469388 55.0723111996961\\
57.7551020408163 55.0723111996961\\
59.5918367346939 55.0723111996961\\
61.4285714285714 55.0723111996961\\
63.265306122449 55.0723111996961\\
65.1020408163265 55.0723111996961\\
66.9387755102041 55.0723111996961\\
68.7755102040816 55.0723111996961\\
70.6122448979592 55.0723111996961\\
72.4489795918367 55.0723111996961\\
74.2857142857143 55.0723111996961\\
76.1224489795918 55.0723111996961\\
77.9591836734694 55.0723111996961\\
79.7959183673469 55.0723111996961\\
81.6326530612245 55.0723111996961\\
83.469387755102 55.0723111996961\\
85.3061224489796 55.0723111996961\\
87.1428571428571 55.0723111996961\\
88.9795918367347 55.0723111996961\\
90.8163265306122 55.0723111996961\\
92.6530612244898 55.0723111996961\\
94.4897959183673 55.0723111996961\\
96.3265306122449 55.0723111996961\\
98.1632653061225 55.0723111996961\\
100 55.0723111996961\\
};
\addlegendentry{\eqref{eq:Srelax}};

\addplot [
color=green!50!black,
solid,
line width=2.0pt
]
table[row sep=crcr]{
10 11.762626289009\\
11.8367346938776 11.7626262890091\\
13.6734693877551 11.7626262890091\\
15.5102040816327 11.7626262890091\\
17.3469387755102 22.1454410709952\\
19.1836734693878 22.1454410709952\\
21.0204081632653 22.1454410709952\\
22.8571428571429 22.1454410709952\\
24.6938775510204 22.1454410709952\\
26.530612244898 22.1454410709952\\
28.3673469387755 22.1454410709952\\
30.2040816326531 22.1454410709952\\
32.0408163265306 22.1454410709952\\
33.8775510204082 22.1454410709952\\
35.7142857142857 22.1454410709952\\
37.5510204081633 22.1454410709952\\
39.3877551020408 22.1454410709952\\
41.2244897959184 22.1454410709952\\
43.0612244897959 22.1454410709952\\
44.8979591836735 22.1454410709952\\
46.734693877551 55.1383445284065\\
48.5714285714286 55.1383445284065\\
50.4081632653061 55.1383445284065\\
52.2448979591837 55.1383445284065\\
54.0816326530612 55.1383445284065\\
55.9183673469388 55.1383445284065\\
57.7551020408163 55.1383445284065\\
59.5918367346939 55.1383445284065\\
61.4285714285714 55.1383445284065\\
63.265306122449 55.1383445284065\\
65.1020408163265 55.1383445284065\\
66.9387755102041 55.1383445284065\\
68.7755102040816 55.1383445284065\\
70.6122448979592 55.1383445284065\\
72.4489795918367 55.1383445284065\\
74.2857142857143 55.1383445284065\\
76.1224489795918 55.1383445284065\\
77.9591836734694 55.1383445284065\\
79.7959183673469 55.1383445284065\\
81.6326530612245 55.1383445284065\\
83.469387755102 55.1383445284065\\
85.3061224489796 55.1383445284065\\
87.1428571428571 55.1383445284065\\
88.9795918367347 55.1383445284065\\
90.8163265306122 55.1383445284065\\
92.6530612244898 55.1383445284065\\
94.4897959183673 55.1383445284065\\
96.3265306122449 55.1383445284065\\
98.1632653061225 55.1383445284065\\
100 55.1383445284065\\
};
\addlegendentry{\eqref{eq:crossover}};

\addplot [
color=red,
solid,
line width=2.0pt
]
table[row sep=crcr]{
10 22.7617835664125\\
11.8367346938776 26.0211544385053\\
13.6734693877551 29.3356620272143\\
15.5102040816327 32.1564142817628\\
17.3469387755102 34.2177027294826\\
19.1836734693878 36.3606421505411\\
21.0204081632653 38.5703717131744\\
22.8571428571429 40.8350207670937\\
24.6938775510204 43.145092985927\\
26.530612244898 45.4929651030055\\
28.3673469387755 47.872488617937\\
30.2040816326531 50.2786774184008\\
32.0408163265306 52.7074643106233\\
33.8775510204082 55.1555116243782\\
35.7142857142857 57.6200637555786\\
37.5510204081633 60.0988320516436\\
39.3877551020408 62.5899045975293\\
41.2244897959184 65.0916751888079\\
43.0612244897959 67.6027871263302\\
44.8979591836735 70.122088501607\\
46.734693877551 72.6485964291448\\
48.5714285714286 75.1814682783806\\
50.4081632653061 76.6049996965669\\
52.2448979591837 77.9419336911464\\
54.0816326530612 79.3008802947928\\
55.9183673469388 80.6806381127551\\
57.7551020408163 82.0800734371979\\
59.5918367346939 83.4981172654818\\
61.4285714285714 84.9337622500801\\
63.265306122449 86.3860596290275\\
65.1020408163265 87.8541161761493\\
66.9387755102041 89.3370912020101\\
68.7755102040816 90.8341936294655\\
70.6122448979592 92.344679161707\\
72.4489795918367 93.8678475557821\\
74.2857142857143 95.4030400104324\\
76.1224489795918 96.9496366737723\\
77.9591836734694 98.5070542736351\\
79.7959183673469 100.074743871259\\
81.6326530612245 101.652188737312\\
83.469387755102 103.238902347934\\
85.3061224489796 104.834426497523\\
87.1428571428571 106.438329524236\\
88.9795918367347 108.050204643732\\
90.8163265306122 109.669668386268\\
92.6530612244898 111.296359132164\\
94.4897959183673 112.929935740487\\
96.3265306122449 114.570076265876\\
98.1632653061225 116.216476758437\\
100 117.868850141825\\
};
\addlegendentry{\eqref{eq:nuclearrelax}};

\end{axis}
\end{tikzpicture}%
}
		\resizebox{\w}{!}{
%
%
\begin{tikzpicture}

\begin{axis}[%
width=4.52083333333333in,
height=3.565625in,
scale only axis,
xmin=10,
xmax=100,
ymin=0,
ymax=150,
legend style={draw=black,fill=white,legend cell align=left}
]
\addplot [
color=blue,
solid,
line width=2.0pt
]
table[row sep=crcr]{
10 7.99839477463965\\
11.8367346938776 7.12704941096283\\
13.6734693877551 7.04861137182247\\
15.5102040816327 12.5323725787507\\
17.3469387755102 12.6041732598275\\
19.1836734693878 12.5377880178974\\
21.0204081632653 12.6445281018744\\
22.8571428571429 13.1429325986281\\
24.6938775510204 13.7285334456656\\
26.530612244898 13.733043763191\\
28.3673469387755 13.7084776535877\\
30.2040816326531 18.3011755415986\\
32.0408163265306 13.8043727024673\\
33.8775510204082 13.6902401317422\\
35.7142857142857 13.9238499597067\\
37.5510204081633 13.8354869102368\\
39.3877551020408 16.0226519314855\\
41.2244897959184 16.0226519314854\\
43.0612244897959 36.3541552526101\\
44.8979591836735 36.3541552526101\\
46.734693877551 36.3541552526101\\
48.5714285714286 36.3541552526101\\
50.4081632653061 36.3541552526101\\
52.2448979591837 36.3541552526101\\
54.0816326530612 36.3541552526101\\
55.9183673469388 36.3541552526101\\
57.7551020408163 36.3541552526101\\
59.5918367346939 36.3541552526101\\
61.4285714285714 36.3541552526101\\
63.265306122449 36.3541552526101\\
65.1020408163265 36.3541552526101\\
66.9387755102041 36.3541552526101\\
68.7755102040816 36.3541552526101\\
70.6122448979592 36.3541552526101\\
72.4489795918367 36.3541552526101\\
74.2857142857143 36.3541552526101\\
76.1224489795918 36.3541552526101\\
77.9591836734694 36.3541552526101\\
79.7959183673469 36.3541552526101\\
81.6326530612245 36.3541552526101\\
83.469387755102 36.3541552526101\\
85.3061224489796 36.3541552526101\\
87.1428571428571 36.3541552526101\\
88.9795918367347 36.3541552526101\\
90.8163265306122 36.3541552526101\\
92.6530612244898 36.3541552526101\\
94.4897959183673 36.3541552526101\\
96.3265306122449 36.3541552526101\\
98.1632653061225 36.3541552526101\\
100 36.3541552526101\\
};
\addlegendentry{\eqref{eq:Srelax}};

\addplot [
color=green!50!black,
solid,
line width=2.0pt
]
table[row sep=crcr]{
10 10.307760528716\\
11.8367346938776 16.840557164525\\
13.6734693877551 16.8405432535386\\
15.5102040816327 16.8405372615766\\
17.3469387755102 16.8405346091596\\
19.1836734693878 16.8405334068842\\
21.0204081632653 16.8405328849056\\
22.8571428571429 16.84053254172\\
24.6938775510204 16.8405325417022\\
26.530612244898 16.8405325416913\\
28.3673469387755 16.84053254168\\
30.2040816326531 16.8405325416732\\
32.0408163265306 16.8405325416552\\
33.8775510204082 16.8405325416468\\
35.7142857142857 16.8405325416415\\
37.5510204081633 16.8405325416382\\
39.3877551020408 36.5426943504373\\
41.2244897959184 36.5426943504373\\
43.0612244897959 36.5426943504373\\
44.8979591836735 36.5426943504373\\
46.734693877551 36.5426943504373\\
48.5714285714286 36.5426943504373\\
50.4081632653061 36.5426943504373\\
52.2448979591837 36.5426943504373\\
54.0816326530612 36.5426943504373\\
55.9183673469388 36.5426943504372\\
57.7551020408163 36.5426943504373\\
59.5918367346939 36.5426943504373\\
61.4285714285714 36.5426943504373\\
63.265306122449 36.5426943504373\\
65.1020408163265 36.5426943504373\\
66.9387755102041 36.5426943504373\\
68.7755102040816 36.5426943504373\\
70.6122448979592 36.5426943504373\\
72.4489795918367 36.5426943504373\\
74.2857142857143 36.5426943504373\\
76.1224489795918 36.5426943504373\\
77.9591836734694 36.5426943504373\\
79.7959183673469 36.5426943504373\\
81.6326530612245 36.5426943504373\\
83.469387755102 36.5426943504373\\
85.3061224489796 36.5426943504372\\
87.1428571428571 36.5426943504372\\
88.9795918367347 36.5426943504373\\
90.8163265306122 36.5426943504373\\
92.6530612244898 36.5426943504373\\
94.4897959183673 36.5426943504372\\
96.3265306122449 36.5426943504373\\
98.1632653061225 36.5426943504373\\
100 36.5426943504373\\
};
\addlegendentry{\eqref{eq:crossover}};

\addplot [
color=red,
solid,
line width=2.0pt
]
table[row sep=crcr]{
10 24.4551089887422\\
11.8367346938776 28.2088042091209\\
13.6734693877551 31.3495630831929\\
15.5102040816327 34.3658729370791\\
17.3469387755102 37.4277880234445\\
19.1836734693878 40.5148659259673\\
21.0204081632653 43.6119737160722\\
22.8571428571429 46.7076554811422\\
24.6938775510204 49.792784618906\\
26.530612244898 52.3837837461698\\
28.3673469387755 54.2560847508776\\
30.2040816326531 56.1807612263242\\
32.0408163265306 58.1523561516002\\
33.8775510204082 60.1659882943357\\
35.7142857142857 62.2172923870356\\
37.5510204081633 64.302362168409\\
39.3877551020408 66.4176973856702\\
41.2244897959184 68.5601552694477\\
43.0612244897959 70.7269066010153\\
44.8979591836735 72.9153962417216\\
46.734693877551 75.1233078458393\\
48.5714285714286 77.3485323991194\\
50.4081632653061 79.5891401928476\\
52.2448979591837 81.8433558405817\\
54.0816326530612 84.1095359605867\\
55.9183673469388 86.3861491734751\\
57.7551020408163 88.6717580968579\\
59.5918367346939 90.9650030538988\\
61.4285714285714 93.2645872489295\\
63.265306122449 95.5692632001914\\
65.1020408163265 97.877820257329\\
66.9387755102041 100.18907307021\\
68.7755102040816 102.501850916794\\
70.6122448979592 104.814987842474\\
72.4489795918367 107.127313612664\\
74.2857142857143 109.437645535671\\
76.1224489795918 111.744781274789\\
77.9591836734694 114.047492836964\\
79.7959183673469 116.344521999191\\
81.6326530612245 118.634577509429\\
83.469387755102 120.916334470545\\
85.3061224489796 123.188436373826\\
87.1428571428571 125.449500279653\\
88.9795918367347 127.698125629733\\
90.8163265306122 129.93290709799\\
92.6530612244898 132.152451726641\\
94.4897959183673 134.355400336141\\
96.3265306122449 136.540452840813\\
98.1632653061225 138.706396663404\\
100 140.85213696416\\
};
\addlegendentry{\eqref{eq:nuclearrelax}};

\end{axis}
\end{tikzpicture}%
}\\
		\resizebox{\w}{!}{
%
%
\begin{tikzpicture}

\begin{axis}[%
width=4.52083333333333in,
height=3.565625in,
scale only axis,
xmin=10,
xmax=100,
ymin=0,
ymax=250,
legend style={draw=black,fill=white,legend cell align=left}
]
\addplot [
color=blue,
solid,
line width=2.0pt
]
table[row sep=crcr]{
10 238.07220567951\\
11.8367346938776 238.186352225344\\
13.6734693877551 238.054206857414\\
15.5102040816327 240.398247740578\\
17.3469387755102 212.529203463975\\
19.1836734693878 179.179550235023\\
21.0204081632653 130.716223434094\\
22.8571428571429 111.614056869343\\
24.6938775510204 110.950970131828\\
26.530612244898 112.625452428149\\
28.3673469387755 111.825734759923\\
30.2040816326531 110.879526992931\\
32.0408163265306 110.753850623394\\
33.8775510204082 110.565782395583\\
35.7142857142857 110.467047166674\\
37.5510204081633 112.095031147929\\
39.3877551020408 43.4688285234344\\
41.2244897959184 22.8451376453004\\
43.0612244897959 22.8451376453004\\
44.8979591836735 56.4081606866325\\
46.734693877551 56.4081606866325\\
48.5714285714286 56.4081606866325\\
50.4081632653061 56.4081606866325\\
52.2448979591837 56.4081606866325\\
54.0816326530612 56.4081606866325\\
55.9183673469388 56.4081606866325\\
57.7551020408163 56.4081606866325\\
59.5918367346939 56.4081606866325\\
61.4285714285714 56.4081606866325\\
63.265306122449 56.4081606866325\\
65.1020408163265 56.4081606866325\\
66.9387755102041 56.4081606866325\\
68.7755102040816 56.4081606866325\\
70.6122448979592 56.4081606866325\\
72.4489795918367 56.4081606866325\\
74.2857142857143 56.4081606866325\\
76.1224489795918 56.4081606866325\\
77.9591836734694 56.4081606866325\\
79.7959183673469 56.4081606866325\\
81.6326530612245 56.4081606866325\\
83.469387755102 56.4081606866325\\
85.3061224489796 56.4081606866325\\
87.1428571428571 56.4081606866325\\
88.9795918367347 56.4081606866325\\
90.8163265306122 56.4081606866325\\
92.6530612244898 56.4081606866325\\
94.4897959183673 56.4081606866325\\
96.3265306122449 56.4081606866325\\
98.1632653061225 56.4081606866325\\
100 56.4081606866325\\
};
\addlegendentry{\eqref{eq:Srelax}};

\addplot [
color=green!50!black,
solid,
line width=2.0pt
]
table[row sep=crcr]{
10 18.599379267416\\
11.8367346938776 23.2688435602151\\
13.6734693877551 23.2688435602151\\
15.5102040816327 23.2688435602152\\
17.3469387755102 23.2688435602151\\
19.1836734693878 23.2688435602151\\
21.0204081632653 23.2688435602152\\
22.8571428571429 23.2688435602152\\
24.6938775510204 23.2688435602152\\
26.530612244898 23.2688435602151\\
28.3673469387755 23.2688435602151\\
30.2040816326531 23.2688435602152\\
32.0408163265306 23.2688435602151\\
33.8775510204082 23.2688435602151\\
35.7142857142857 23.2688435602151\\
37.5510204081633 23.2688435602152\\
39.3877551020408 23.2688435602152\\
41.2244897959184 23.2688435602151\\
43.0612244897959 56.4863178686485\\
44.8979591836735 56.4863178686485\\
46.734693877551 56.4863178686485\\
48.5714285714286 56.4863178686485\\
50.4081632653061 56.4863178686485\\
52.2448979591837 56.4863178686485\\
54.0816326530612 56.4863178686485\\
55.9183673469388 56.4863178686485\\
57.7551020408163 56.4863178686485\\
59.5918367346939 56.4863178686485\\
61.4285714285714 56.4863178686485\\
63.265306122449 56.4863178686485\\
65.1020408163265 56.4863178686485\\
66.9387755102041 56.4863178686485\\
68.7755102040816 56.4863178686485\\
70.6122448979592 56.4863178686485\\
72.4489795918367 56.4863178686485\\
74.2857142857143 56.4863178686485\\
76.1224489795918 56.4863178686485\\
77.9591836734694 56.4863178686485\\
79.7959183673469 56.4863178686485\\
81.6326530612245 56.4863178686485\\
83.469387755102 56.4863178686485\\
85.3061224489796 56.4863178686485\\
87.1428571428571 56.4863178686485\\
88.9795918367347 56.4863178686485\\
90.8163265306122 56.4863178686485\\
92.6530612244898 56.4863178686485\\
94.4897959183673 56.4863178686485\\
96.3265306122449 56.4863178686485\\
98.1632653061225 56.4863178686485\\
100 56.4863178686485\\
};
\addlegendentry{\eqref{eq:crossover}};

\addplot [
color=red,
solid,
line width=2.0pt
]
table[row sep=crcr]{
10 28.396533934947\\
11.8367346938776 31.0484101070555\\
13.6734693877551 33.2175558141678\\
15.5102040816327 35.4937090321627\\
17.3469387755102 37.8523177223975\\
19.1836734693878 40.2747943779386\\
21.0204081632653 42.7470321376581\\
22.8571428571429 45.258271719294\\
24.6938775510204 47.8002551108636\\
26.530612244898 50.3665997037784\\
28.3673469387755 52.9523367225948\\
30.2040816326531 55.553570281567\\
32.0408163265306 58.1672244815818\\
33.8775510204082 60.7908547299016\\
35.7142857142857 63.4225060374043\\
37.5510204081633 66.0606058420266\\
39.3877551020408 68.7038823580005\\
41.2244897959184 71.3513019215768\\
43.0612244897959 74.0020205723009\\
44.8979591836735 76.3738530994683\\
46.734693877551 77.7755744118573\\
48.5714285714286 79.2036607724644\\
50.4081632653061 80.6565016576529\\
52.2448979591837 82.1325818347974\\
54.0816326530612 83.630477321341\\
55.9183673469388 85.1488511361437\\
57.7551020408163 86.6864489422152\\
59.5918367346939 88.2420946606127\\
61.4285714285714 89.8146861185635\\
63.265306122449 91.4031907806464\\
65.1020408163265 93.0066415999544\\
66.9387755102041 94.6241330161828\\
68.7755102040816 96.2548171194604\\
70.6122448979592 97.8978999921482\\
72.4489795918367 99.552638235463\\
74.2857142857143 101.218335683672\\
76.1224489795918 102.894340305291\\
77.9591836734694 104.580041288208\\
79.7959183673469 106.274866303889\\
81.6326530612245 107.978278944306\\
83.469387755102 109.689776324485\\
85.3061224489796 111.408886842809\\
87.1428571428571 113.135168090941\\
88.9795918367347 114.868204905102\\
90.8163265306122 116.607607550414\\
92.6530612244898 118.353010030155\\
94.4897959183673 120.104068512059\\
96.3265306122449 121.860459863993\\
98.1632653061225 123.621880291708\\
100 125.388044071768\\
};
\addlegendentry{\eqref{eq:nuclearrelax}};

\end{axis}
\end{tikzpicture}%
}
		\resizebox{\w}{!}{
%
%
\begin{tikzpicture}

\begin{axis}[%
width=4.52083333333333in,
height=3.565625in,
scale only axis,
xmin=10,
xmax=100,
ymin=0,
ymax=140,
legend style={draw=black,fill=white,legend cell align=left}
]
\addplot [
color=blue,
solid,
line width=2.0pt
]
table[row sep=crcr]{
10 85.9124487830642\\
11.8367346938776 73.0731614146766\\
13.6734693877551 59.0327524584651\\
15.5102040816327 58.5360633040702\\
17.3469387755102 58.9259519742874\\
19.1836734693878 58.0007529437822\\
21.0204081632653 33.7481459581536\\
22.8571428571429 34.4234297684639\\
24.6938775510204 12.4770498369827\\
26.530612244898 12.4770498381915\\
28.3673469387755 36.397827536368\\
30.2040816326531 36.397827536368\\
32.0408163265306 36.397827536368\\
33.8775510204082 36.397827536368\\
35.7142857142857 56.4339484691543\\
37.5510204081633 56.4339484691543\\
39.3877551020408 56.4339484691543\\
41.2244897959184 56.4339484691543\\
43.0612244897959 56.4339484691543\\
44.8979591836735 56.4339484691543\\
46.734693877551 56.4339484691543\\
48.5714285714286 56.4339484691543\\
50.4081632653061 56.4339484691543\\
52.2448979591837 56.4339484691543\\
54.0816326530612 56.4339484691543\\
55.9183673469388 56.4339484691543\\
57.7551020408163 56.4339484691543\\
59.5918367346939 56.4339484691543\\
61.4285714285714 56.4339484691543\\
63.265306122449 56.4339484691543\\
65.1020408163265 56.4339484691543\\
66.9387755102041 56.4339484691543\\
68.7755102040816 56.4339484691543\\
70.6122448979592 56.4339484691543\\
72.4489795918367 56.4339484691543\\
74.2857142857143 56.4339484691543\\
76.1224489795918 56.4339484691543\\
77.9591836734694 56.4339484691543\\
79.7959183673469 56.4339484691543\\
81.6326530612245 56.4339484691543\\
83.469387755102 56.4339484691543\\
85.3061224489796 56.4339484691543\\
87.1428571428571 56.4339484691543\\
88.9795918367347 56.4339484691543\\
90.8163265306122 56.4339484691543\\
92.6530612244898 56.4339484691543\\
94.4897959183673 56.4339484691543\\
96.3265306122449 56.4339484691543\\
98.1632653061225 56.4339484691543\\
100 56.4339484691543\\
};
\addlegendentry{\eqref{eq:Srelax}};

\addplot [
color=green!50!black,
solid,
line width=2.0pt
]
table[row sep=crcr]{
10 15.6336166948269\\
11.8367346938776 15.6336166948269\\
13.6734693877551 15.6336166948269\\
15.5102040816327 15.6336166948269\\
17.3469387755102 15.6336166948268\\
19.1836734693878 36.4820968148001\\
21.0204081632653 36.4820968148001\\
22.8571428571429 36.4820968148001\\
24.6938775510204 36.4820968148002\\
26.530612244898 36.4820968148001\\
28.3673469387755 36.4820968148002\\
30.2040816326531 36.4820968148002\\
32.0408163265306 36.4820968148002\\
33.8775510204082 56.511074661353\\
35.7142857142857 56.511074661353\\
37.5510204081633 56.511074661353\\
39.3877551020408 56.511074661353\\
41.2244897959184 56.511074661353\\
43.0612244897959 56.511074661353\\
44.8979591836735 56.511074661353\\
46.734693877551 56.511074661353\\
48.5714285714286 56.511074661353\\
50.4081632653061 56.511074661353\\
52.2448979591837 56.511074661353\\
54.0816326530612 56.511074661353\\
55.9183673469388 56.511074661353\\
57.7551020408163 56.511074661353\\
59.5918367346939 56.511074661353\\
61.4285714285714 56.511074661353\\
63.265306122449 56.511074661353\\
65.1020408163265 56.511074661353\\
66.9387755102041 56.511074661353\\
68.7755102040816 56.511074661353\\
70.6122448979592 56.511074661353\\
72.4489795918367 56.511074661353\\
74.2857142857143 56.511074661353\\
76.1224489795918 56.511074661353\\
77.9591836734694 56.511074661353\\
79.7959183673469 56.511074661353\\
81.6326530612245 56.511074661353\\
83.469387755102 56.511074661353\\
85.3061224489796 56.511074661353\\
87.1428571428571 56.511074661353\\
88.9795918367347 56.511074661353\\
90.8163265306122 56.511074661353\\
92.6530612244898 56.511074661353\\
94.4897959183673 56.511074661353\\
96.3265306122449 56.511074661353\\
98.1632653061225 56.511074661353\\
100 56.511074661353\\
};
\addlegendentry{\eqref{eq:crossover}};

\addplot [
color=red,
solid,
line width=2.0pt
]
table[row sep=crcr]{
10 31.4889653445391\\
11.8367346938776 35.3203830095219\\
13.6734693877551 39.0333614882839\\
15.5102040816327 42.6310531875104\\
17.3469387755102 46.1304221013147\\
19.1836734693878 49.552483520478\\
21.0204081632653 52.5475068855043\\
22.8571428571429 54.8037982757057\\
24.6938775510204 57.1098524243998\\
26.530612244898 59.4576978072864\\
28.3673469387755 61.8408059690448\\
30.2040816326531 64.2538071413671\\
32.0408163265306 66.692263796648\\
33.8775510204082 69.152490716195\\
35.7142857142857 71.5371771135256\\
37.5510204081633 72.8924282654904\\
39.3877551020408 74.2811313835866\\
41.2244897959184 75.7008771691043\\
43.0612244897959 77.1494157448662\\
44.8979591836735 78.62464997145\\
46.734693877551 80.1246281096378\\
48.5714285714286 81.6475360933244\\
50.4081632653061 83.1916896248509\\
52.2448979591837 84.7555262592686\\
54.0816326530612 86.3375976053338\\
55.9183673469388 87.9365617386275\\
57.7551020408163 89.5511758955226\\
59.5918367346939 91.1802894950915\\
61.4285714285714 92.8228375187781\\
63.265306122449 94.4778342640631\\
65.1020408163265 96.1443674778494\\
66.9387755102041 97.8215928673553\\
68.7755102040816 99.5087289803646\\
70.6122448979592 101.205052442481\\
72.4489795918367 102.909893536105\\
74.2857142857143 104.622632103957\\
76.1224489795918 106.342693758927\\
77.9591836734694 108.069546381513\\
79.7959183673469 109.802696886173\\
81.6326530612245 111.541688238194\\
83.469387755102 113.286096703297\\
85.3061224489796 115.035529312917\\
87.1428571428571 116.789621528973\\
88.9795918367347 118.548035092832\\
90.8163265306122 120.310456044145\\
92.6530612244898 122.076592896148\\
94.4897959183673 123.846174954979\\
96.3265306122449 125.618950771444\\
98.1632653061225 127.394686714557\\
100 129.173165656951\\
};
\addlegendentry{\eqref{eq:nuclearrelax}};

\end{axis}
\end{tikzpicture}%
}
		\resizebox{\w}{!}{
%
%
\begin{tikzpicture}

\begin{axis}[%
width=4.52083333333333in,
height=3.565625in,
scale only axis,
xmin=10,
xmax=100,
ymin=0,
ymax=140,
legend style={draw=black,fill=white,legend cell align=left}
]
\addplot [
color=blue,
solid,
line width=2.0pt
]
table[row sep=crcr]{
10 83.1589896954453\\
11.8367346938776 70.4997353523318\\
13.6734693877551 64.7731657005461\\
15.5102040816327 65.1443990717615\\
17.3469387755102 64.1100840715674\\
19.1836734693878 37.3834173153666\\
21.0204081632653 32.8073037221049\\
22.8571428571429 34.1096789839603\\
24.6938775510204 17.1075874033141\\
26.530612244898 29.6471009868828\\
28.3673469387755 29.6471009868828\\
30.2040816326531 29.6471009868828\\
32.0408163265306 29.6471009868828\\
33.8775510204082 29.6471009868828\\
35.7142857142857 29.6471009868828\\
37.5510204081633 29.6471009868828\\
39.3877551020408 29.6471009868828\\
41.2244897959184 29.6471009868828\\
43.0612244897959 29.6471009868828\\
44.8979591836735 29.6471009868828\\
46.734693877551 29.6471009868828\\
48.5714285714286 29.6471009868828\\
50.4081632653061 61.6574798421365\\
52.2448979591837 61.6574798421365\\
54.0816326530612 61.6574798421365\\
55.9183673469388 61.6574798421365\\
57.7551020408163 61.6574798421365\\
59.5918367346939 61.6574798421365\\
61.4285714285714 61.6574798421365\\
63.265306122449 61.6574798421365\\
65.1020408163265 61.6574798421365\\
66.9387755102041 61.6574798421366\\
68.7755102040816 61.6574798421366\\
70.6122448979592 61.6574798421365\\
72.4489795918367 61.6574798421365\\
74.2857142857143 61.6574798421366\\
76.1224489795918 61.6574798421365\\
77.9591836734694 61.6574798421365\\
79.7959183673469 61.6574798421366\\
81.6326530612245 61.6574798421365\\
83.469387755102 61.6574798421365\\
85.3061224489796 61.6574798421365\\
87.1428571428571 61.6574798421365\\
88.9795918367347 61.6574798421365\\
90.8163265306122 61.6574798421365\\
92.6530612244898 61.6574798421365\\
94.4897959183673 61.6574798421366\\
96.3265306122449 61.6574798421365\\
98.1632653061225 61.6574798421365\\
100 61.6574798421365\\
};
\addlegendentry{\eqref{eq:Srelax}};

\addplot [
color=green!50!black,
solid,
line width=2.0pt
]
table[row sep=crcr]{
10 17.7976182310609\\
11.8367346938776 17.7976182310609\\
13.6734693877551 17.797618231061\\
15.5102040816327 17.7976182310609\\
17.3469387755102 29.9630697377042\\
19.1836734693878 29.9630697377042\\
21.0204081632653 29.9630697377042\\
22.8571428571429 29.9630697377042\\
24.6938775510204 29.9630697377042\\
26.530612244898 29.9630697377042\\
28.3673469387755 29.9630697377042\\
30.2040816326531 29.9630697377042\\
32.0408163265306 29.9630697377042\\
33.8775510204082 29.9630697377042\\
35.7142857142857 29.9630697377042\\
37.5510204081633 29.9630697377042\\
39.3877551020408 29.9630697377042\\
41.2244897959184 29.9630697377042\\
43.0612244897959 29.9630697377042\\
44.8979591836735 29.9630697377042\\
46.734693877551 61.7213455172719\\
48.5714285714286 61.7213455172719\\
50.4081632653061 61.7213455172719\\
52.2448979591837 61.7213455172719\\
54.0816326530612 61.7213455172719\\
55.9183673469388 61.7213455172719\\
57.7551020408163 61.7213455172719\\
59.5918367346939 61.7213455172719\\
61.4285714285714 61.7213455172719\\
63.265306122449 61.7213455172719\\
65.1020408163265 61.7213455172719\\
66.9387755102041 61.7213455172719\\
68.7755102040816 61.7213455172719\\
70.6122448979592 61.7213455172719\\
72.4489795918367 61.7213455172719\\
74.2857142857143 61.7213455172719\\
76.1224489795918 61.7213455172719\\
77.9591836734694 61.7213455172719\\
79.7959183673469 61.7213455172719\\
81.6326530612245 61.7213455172719\\
83.469387755102 61.7213455172719\\
85.3061224489796 61.7213455172719\\
87.1428571428571 61.7213455172719\\
88.9795918367347 61.7213455172719\\
90.8163265306122 61.7213455172719\\
92.6530612244898 61.7213455172719\\
94.4897959183673 61.7213455172719\\
96.3265306122449 61.7213455172719\\
98.1632653061225 61.7213455172719\\
100 61.7213455172719\\
};
\addlegendentry{\eqref{eq:crossover}};

\addplot [
color=red,
solid,
line width=2.0pt
]
table[row sep=crcr]{
10 29.3152958089032\\
11.8367346938776 32.990663459696\\
13.6734693877551 36.7527359890847\\
15.5102040816327 39.8502182055139\\
17.3469387755102 41.9150492164664\\
19.1836734693878 44.0682410994633\\
21.0204081632653 46.2942460986734\\
22.8571428571429 48.5803971184943\\
24.6938775510204 50.916387619155\\
26.530612244898 53.2938260784583\\
28.3673469387755 55.7058660505324\\
30.2040816326531 58.1469048006527\\
32.0408163265306 60.6123406057482\\
33.8775510204082 63.0983785628161\\
35.7142857142857 65.6018756911756\\
37.5510204081633 68.12021747333\\
39.3877551020408 70.6512193705433\\
41.2244897959184 73.1930481057974\\
43.0612244897959 75.7441585726112\\
44.8979591836735 78.3032431005981\\
46.734693877551 80.8691905076506\\
48.5714285714286 83.4410529215069\\
50.4081632653061 84.9408038286644\\
52.2448979591837 86.3566909214551\\
54.0816326530612 87.7935450139954\\
55.9183673469388 89.2501057728672\\
57.7551020408163 90.7251900534778\\
59.5918367346939 92.2176878355833\\
61.4285714285714 93.7265582050632\\
63.265306122449 95.2508254204562\\
65.1020408163265 96.7895750935331\\
66.9387755102041 98.3419505054402\\
68.7755102040816 99.9071490735402\\
70.6122448979592 101.484418978809\\
72.4489795918367 103.07305595939\\
74.2857142857143 104.672400272517\\
76.1224489795918 106.281833824294\\
77.9591836734694 107.900777464809\\
79.7959183673469 109.528688444392\\
81.6326530612245 111.16505802577\\
83.469387755102 112.809409245949\\
85.3061224489796 114.461294821177\\
87.1428571428571 116.120295187976\\
88.9795918367347 117.786016673068\\
90.8163265306122 119.458089785018\\
92.6530612244898 121.136167620486\\
94.4897959183673 122.819924378135\\
96.3265306122449 124.509053973475\\
98.1632653061225 126.20326874816\\
100 127.902298267575\\
};
\addlegendentry{\eqref{eq:nuclearrelax}};

\end{axis}
\end{tikzpicture}%
}
		\resizebox{\w}{!}{
%
%
\begin{tikzpicture}

\begin{axis}[%
width=4.52083333333333in,
height=3.565625in,
scale only axis,
xmin=10,
xmax=100,
ymin=0,
ymax=200,
legend style={draw=black,fill=white,legend cell align=left}
]
\addplot [
color=blue,
solid,
line width=2.0pt
]
table[row sep=crcr]{
10 146.701191988413\\
11.8367346938776 135.454970790107\\
13.6734693877551 134.458733084621\\
15.5102040816327 125.940869492342\\
17.3469387755102 126.89554292471\\
19.1836734693878 127.583927583223\\
21.0204081632653 125.908744935466\\
22.8571428571429 127.106684983785\\
24.6938775510204 57.1197356369891\\
26.530612244898 61.7933029661512\\
28.3673469387755 56.2450942330734\\
30.2040816326531 175.2785162551\\
32.0408163265306 59.899816809771\\
33.8775510204082 55.2972337107543\\
35.7142857142857 59.5945913730412\\
37.5510204081633 57.3550412504187\\
39.3877551020408 23.6948994122284\\
41.2244897959184 23.6948994049792\\
43.0612244897959 51.1504771994437\\
44.8979591836735 51.1504771994437\\
46.734693877551 51.1504771994437\\
48.5714285714286 51.1504771994437\\
50.4081632653061 51.1504771994437\\
52.2448979591837 51.1504771994437\\
54.0816326530612 51.1504771994437\\
55.9183673469388 51.1504771994436\\
57.7551020408163 51.1504771994437\\
59.5918367346939 51.1504771994436\\
61.4285714285714 51.1504771994437\\
63.265306122449 51.1504771994437\\
65.1020408163265 51.1504771994437\\
66.9387755102041 51.1504771994436\\
68.7755102040816 51.1504771994437\\
70.6122448979592 51.1504771994437\\
72.4489795918367 51.1504771994437\\
74.2857142857143 51.1504771994437\\
76.1224489795918 51.1504771994437\\
77.9591836734694 51.1504771994436\\
79.7959183673469 51.1504771994436\\
81.6326530612245 51.1504771994437\\
83.469387755102 51.1504771994436\\
85.3061224489796 51.1504771994437\\
87.1428571428571 51.1504771994437\\
88.9795918367347 51.1504771994437\\
90.8163265306122 51.1504771994436\\
92.6530612244898 51.1504771994436\\
94.4897959183673 51.1504771994436\\
96.3265306122449 51.1504771994437\\
98.1632653061225 51.1504771994436\\
100 51.1504771994437\\
};
\addlegendentry{\eqref{eq:Srelax}};

\addplot [
color=green!50!black,
solid,
line width=2.0pt
]
table[row sep=crcr]{
10 17.6250057016165\\
11.8367346938776 24.9672420962091\\
13.6734693877551 24.9671315944634\\
15.5102040816327 24.9670835888518\\
17.3469387755102 24.9670621688261\\
19.1836734693878 24.9670523954193\\
21.0204081632653 24.9670481300137\\
22.8571428571429 24.9670452920515\\
24.6938775510204 24.9670452919023\\
26.530612244898 24.9670452918105\\
28.3673469387755 24.9670452917166\\
30.2040816326531 24.9670452916594\\
32.0408163265306 24.9670452915083\\
33.8775510204082 24.9670452914382\\
35.7142857142857 24.9670452913933\\
37.5510204081633 24.9670452913666\\
39.3877551020408 51.366298231503\\
41.2244897959184 51.3662982315031\\
43.0612244897959 51.3662982315031\\
44.8979591836735 51.3662982315031\\
46.734693877551 51.3662982315031\\
48.5714285714286 51.3662982315031\\
50.4081632653061 51.3662982315031\\
52.2448979591837 51.3662982315031\\
54.0816326530612 51.3662982315031\\
55.9183673469388 51.366298231503\\
57.7551020408163 51.3662982315031\\
59.5918367346939 51.3662982315031\\
61.4285714285714 51.3662982315031\\
63.265306122449 51.3662982315031\\
65.1020408163265 51.3662982315031\\
66.9387755102041 51.3662982315031\\
68.7755102040816 51.3662982315031\\
70.6122448979592 51.3662982315031\\
72.4489795918367 51.3662982315031\\
74.2857142857143 51.3662982315031\\
76.1224489795918 51.3662982315031\\
77.9591836734694 51.3662982315031\\
79.7959183673469 51.3662982315031\\
81.6326530612245 51.366298231503\\
83.469387755102 51.366298231503\\
85.3061224489796 51.366298231503\\
87.1428571428571 51.3662982315029\\
88.9795918367347 51.3662982315031\\
90.8163265306122 51.3662982315031\\
92.6530612244898 51.366298231503\\
94.4897959183673 51.366298231503\\
96.3265306122449 51.3662982315031\\
98.1632653061225 51.366298231503\\
100 51.3662982315031\\
};
\addlegendentry{\eqref{eq:crossover}};

\addplot [
color=red,
solid,
line width=2.0pt
]
table[row sep=crcr]{
10 37.3327615147751\\
11.8367346938776 42.2872889016743\\
13.6734693877551 46.6449527390173\\
15.5102040816327 50.8801456429384\\
17.3469387755102 55.1332713996943\\
19.1836734693878 59.3792725020445\\
21.0204081632653 63.5991378108806\\
22.8571428571429 67.7779441484131\\
24.6938775510204 71.9029661463728\\
26.530612244898 75.4307064896592\\
28.3673469387755 78.1463808541426\\
30.2040816326531 80.926780969473\\
32.0408163265306 83.7650929349002\\
33.8775510204082 86.6552590069972\\
35.7142857142857 89.5918797132161\\
37.5510204081633 92.5701275591445\\
39.3877551020408 95.5856711480659\\
41.2244897959184 98.6346085421453\\
43.0612244897959 101.713408741865\\
44.8979591836735 104.818860232766\\
46.734693877551 107.948025635256\\
48.5714285714286 111.098201585657\\
50.4081632653061 114.266883068795\\
52.2448979591837 117.451731510298\\
54.0816326530612 120.650546018651\\
55.9183673469388 123.861237241925\\
57.7551020408163 127.081803372444\\
59.5918367346939 130.310307895394\\
61.4285714285714 133.54485873596\\
63.265306122449 136.783588516168\\
65.1020408163265 140.024635689286\\
66.9387755102041 143.266126379514\\
68.7755102040816 146.506156820589\\
70.6122448979592 149.742776362465\\
72.4489795918367 152.973971103954\\
74.2857142857143 156.197648314616\\
76.1224489795918 159.411621934359\\
77.9591836734694 162.613599585603\\
79.7959183673469 165.801171699864\\
81.6326530612245 168.971803543304\\
83.469387755102 172.122831113642\\
85.3061224489796 175.25146205504\\
87.1428571428571 178.354782869999\\
88.9795918367347 181.42977375912\\
90.8163265306122 184.473332343112\\
92.6530612244898 187.482307265708\\
94.4897959183673 190.453542196428\\
96.3265306122449 193.383930024851\\
98.1632653061225 196.270476079608\\
100 199.110368086938\\
};
\addlegendentry{\eqref{eq:nuclearrelax}};

\end{axis}
\end{tikzpicture}%
}\\
		\resizebox{\w}{!}{
%
%
\begin{tikzpicture}

\begin{axis}[%
width=4.52083333333333in,
height=3.565625in,
scale only axis,
xmin=10,
xmax=100,
ymin=0,
ymax=14,
legend style={draw=black,fill=white,legend cell align=left}
]
\addplot [
color=blue,
solid,
line width=2.0pt
]
table[row sep=crcr]{
10 13.6666666666667\\
11.8367346938776 13.6666666666667\\
13.6734693877551 13.6666666666667\\
15.5102040816327 13.6666666666667\\
17.3469387755102 13.6666666666667\\
19.1836734693878 13.6666666666667\\
21.0204081632653 12\\
22.8571428571429 4\\
24.6938775510204 4\\
26.530612244898 4\\
28.3673469387755 4\\
30.2040816326531 4\\
32.0408163265306 4\\
33.8775510204082 4\\
35.7142857142857 4\\
37.5510204081633 4\\
39.3877551020408 3\\
41.2244897959184 2\\
43.0612244897959 2\\
44.8979591836735 1\\
46.734693877551 1\\
48.5714285714286 1\\
50.4081632653061 1\\
52.2448979591837 1\\
54.0816326530612 1\\
55.9183673469388 1\\
57.7551020408163 1\\
59.5918367346939 1\\
61.4285714285714 1\\
63.265306122449 1\\
65.1020408163265 1\\
66.9387755102041 1\\
68.7755102040816 1\\
70.6122448979592 1\\
72.4489795918367 1\\
74.2857142857143 1\\
76.1224489795918 1\\
77.9591836734694 1\\
79.7959183673469 1\\
81.6326530612245 1\\
83.469387755102 1\\
85.3061224489796 1\\
87.1428571428571 1\\
88.9795918367347 1\\
90.8163265306122 1\\
92.6530612244898 1\\
94.4897959183673 1\\
96.3265306122449 1\\
98.1632653061225 1\\
100 1\\
};
\addlegendentry{\eqref{eq:Srelax}};

\addplot [
color=green!50!black,
solid,
line width=2.0pt
]
table[row sep=crcr]{
10 3\\
11.8367346938776 2\\
13.6734693877551 2\\
15.5102040816327 2\\
17.3469387755102 2\\
19.1836734693878 2\\
21.0204081632653 2\\
22.8571428571429 2\\
24.6938775510204 2\\
26.530612244898 2\\
28.3673469387755 2\\
30.2040816326531 2\\
32.0408163265306 2\\
33.8775510204082 2\\
35.7142857142857 2\\
37.5510204081633 2\\
39.3877551020408 2\\
41.2244897959184 2\\
43.0612244897959 1\\
44.8979591836735 1\\
46.734693877551 1\\
48.5714285714286 1\\
50.4081632653061 1\\
52.2448979591837 1\\
54.0816326530612 1\\
55.9183673469388 1\\
57.7551020408163 1\\
59.5918367346939 1\\
61.4285714285714 1\\
63.265306122449 1\\
65.1020408163265 1\\
66.9387755102041 1\\
68.7755102040816 1\\
70.6122448979592 1\\
72.4489795918367 1\\
74.2857142857143 1\\
76.1224489795918 1\\
77.9591836734694 1\\
79.7959183673469 1\\
81.6326530612245 1\\
83.469387755102 1\\
85.3061224489796 1\\
87.1428571428571 1\\
88.9795918367347 1\\
90.8163265306122 1\\
92.6530612244898 1\\
94.4897959183673 1\\
96.3265306122449 1\\
98.1632653061225 1\\
100 1\\
};
\addlegendentry{\eqref{eq:crossover}};

\addplot [
color=red,
solid,
line width=2.0pt
]
table[row sep=crcr]{
10 3\\
11.8367346938776 2\\
13.6734693877551 2\\
15.5102040816327 2\\
17.3469387755102 2\\
19.1836734693878 2\\
21.0204081632653 2\\
22.8571428571429 2\\
24.6938775510204 2\\
26.530612244898 2\\
28.3673469387755 2\\
30.2040816326531 2\\
32.0408163265306 2\\
33.8775510204082 2\\
35.7142857142857 2\\
37.5510204081633 2\\
39.3877551020408 2\\
41.2244897959184 2\\
43.0612244897959 2\\
44.8979591836735 1\\
46.734693877551 1\\
48.5714285714286 1\\
50.4081632653061 1\\
52.2448979591837 1\\
54.0816326530612 1\\
55.9183673469388 1\\
57.7551020408163 1\\
59.5918367346939 1\\
61.4285714285714 1\\
63.265306122449 1\\
65.1020408163265 1\\
66.9387755102041 1\\
68.7755102040816 1\\
70.6122448979592 1\\
72.4489795918367 1\\
74.2857142857143 1\\
76.1224489795918 1\\
77.9591836734694 1\\
79.7959183673469 1\\
81.6326530612245 1\\
83.469387755102 1\\
85.3061224489796 1\\
87.1428571428571 1\\
88.9795918367347 1\\
90.8163265306122 1\\
92.6530612244898 1\\
94.4897959183673 1\\
96.3265306122449 1\\
98.1632653061225 1\\
100 1\\
};
\addlegendentry{\eqref{eq:nuclearrelax}};

\end{axis}
\end{tikzpicture}%
} 
		\resizebox{\w}{!}{
%
%
\begin{tikzpicture}

\begin{axis}[%
width=4.52083333333333in,
height=3.565625in,
scale only axis,
xmin=10,
xmax=100,
ymin=0,
ymax=14,
legend style={draw=black,fill=white,legend cell align=left}
]
\addplot [
color=blue,
solid,
line width=2.0pt
]
table[row sep=crcr]{
10 13.3333333333333\\
11.8367346938776 7\\
13.6734693877551 5\\
15.5102040816327 5\\
17.3469387755102 5\\
19.1836734693878 5\\
21.0204081632653 4\\
22.8571428571429 4\\
24.6938775510204 3\\
26.530612244898 3\\
28.3673469387755 2\\
30.2040816326531 2\\
32.0408163265306 2\\
33.8775510204082 2\\
35.7142857142857 1\\
37.5510204081633 1\\
39.3877551020408 1\\
41.2244897959184 1\\
43.0612244897959 1\\
44.8979591836735 1\\
46.734693877551 1\\
48.5714285714286 1\\
50.4081632653061 1\\
52.2448979591837 1\\
54.0816326530612 1\\
55.9183673469388 1\\
57.7551020408163 1\\
59.5918367346939 1\\
61.4285714285714 1\\
63.265306122449 1\\
65.1020408163265 1\\
66.9387755102041 1\\
68.7755102040816 1\\
70.6122448979592 1\\
72.4489795918367 1\\
74.2857142857143 1\\
76.1224489795918 1\\
77.9591836734694 1\\
79.7959183673469 1\\
81.6326530612245 1\\
83.469387755102 1\\
85.3061224489796 1\\
87.1428571428571 1\\
88.9795918367347 1\\
90.8163265306122 1\\
92.6530612244898 1\\
94.4897959183673 1\\
96.3265306122449 1\\
98.1632653061225 1\\
100 1\\
};
\addlegendentry{\eqref{eq:Srelax}};

\addplot [
color=green!50!black,
solid,
line width=2.0pt
]
table[row sep=crcr]{
10 3\\
11.8367346938776 3\\
13.6734693877551 3\\
15.5102040816327 3\\
17.3469387755102 3\\
19.1836734693878 2\\
21.0204081632653 2\\
22.8571428571429 2\\
24.6938775510204 2\\
26.530612244898 2\\
28.3673469387755 2\\
30.2040816326531 2\\
32.0408163265306 2\\
33.8775510204082 1\\
35.7142857142857 1\\
37.5510204081633 1\\
39.3877551020408 1\\
41.2244897959184 1\\
43.0612244897959 1\\
44.8979591836735 1\\
46.734693877551 1\\
48.5714285714286 1\\
50.4081632653061 1\\
52.2448979591837 1\\
54.0816326530612 1\\
55.9183673469388 1\\
57.7551020408163 1\\
59.5918367346939 1\\
61.4285714285714 1\\
63.265306122449 1\\
65.1020408163265 1\\
66.9387755102041 1\\
68.7755102040816 1\\
70.6122448979592 1\\
72.4489795918367 1\\
74.2857142857143 1\\
76.1224489795918 1\\
77.9591836734694 1\\
79.7959183673469 1\\
81.6326530612245 1\\
83.469387755102 1\\
85.3061224489796 1\\
87.1428571428571 1\\
88.9795918367347 1\\
90.8163265306122 1\\
92.6530612244898 1\\
94.4897959183673 1\\
96.3265306122449 1\\
98.1632653061225 1\\
100 1\\
};
\addlegendentry{\eqref{eq:crossover}};

\addplot [
color=red,
solid,
line width=2.0pt
]
table[row sep=crcr]{
10 3\\
11.8367346938776 3\\
13.6734693877551 3\\
15.5102040816327 3\\
17.3469387755102 3\\
19.1836734693878 3\\
21.0204081632653 2\\
22.8571428571429 2\\
24.6938775510204 2\\
26.530612244898 2\\
28.3673469387755 2\\
30.2040816326531 2\\
32.0408163265306 2\\
33.8775510204082 2\\
35.7142857142857 1\\
37.5510204081633 1\\
39.3877551020408 1\\
41.2244897959184 1\\
43.0612244897959 1\\
44.8979591836735 1\\
46.734693877551 1\\
48.5714285714286 1\\
50.4081632653061 1\\
52.2448979591837 1\\
54.0816326530612 1\\
55.9183673469388 1\\
57.7551020408163 1\\
59.5918367346939 1\\
61.4285714285714 1\\
63.265306122449 1\\
65.1020408163265 1\\
66.9387755102041 1\\
68.7755102040816 1\\
70.6122448979592 1\\
72.4489795918367 1\\
74.2857142857143 1\\
76.1224489795918 1\\
77.9591836734694 1\\
79.7959183673469 1\\
81.6326530612245 1\\
83.469387755102 1\\
85.3061224489796 1\\
87.1428571428571 1\\
88.9795918367347 1\\
90.8163265306122 1\\
92.6530612244898 1\\
94.4897959183673 1\\
96.3265306122449 1\\
98.1632653061225 1\\
100 1\\
};
\addlegendentry{\eqref{eq:nuclearrelax}};

\end{axis}
\end{tikzpicture}%
} 
		\resizebox{\w}{!}{
%
%
\begin{tikzpicture}

\begin{axis}[%
width=4.52083333333333in,
height=3.565625in,
scale only axis,
xmin=10,
xmax=100,
ymin=0,
ymax=14,
legend style={draw=black,fill=white,legend cell align=left}
]
\addplot [
color=blue,
solid,
line width=2.0pt
]
table[row sep=crcr]{
10 13.6666666666667\\
11.8367346938776 7\\
13.6734693877551 5\\
15.5102040816327 5\\
17.3469387755102 5\\
19.1836734693878 4\\
21.0204081632653 4\\
22.8571428571429 4\\
24.6938775510204 3\\
26.530612244898 2\\
28.3673469387755 2\\
30.2040816326531 2\\
32.0408163265306 2\\
33.8775510204082 2\\
35.7142857142857 2\\
37.5510204081633 2\\
39.3877551020408 2\\
41.2244897959184 2\\
43.0612244897959 2\\
44.8979591836735 2\\
46.734693877551 2\\
48.5714285714286 2\\
50.4081632653061 1\\
52.2448979591837 1\\
54.0816326530612 1\\
55.9183673469388 1\\
57.7551020408163 1\\
59.5918367346939 1\\
61.4285714285714 1\\
63.265306122449 1\\
65.1020408163265 1\\
66.9387755102041 1\\
68.7755102040816 1\\
70.6122448979592 1\\
72.4489795918367 1\\
74.2857142857143 1\\
76.1224489795918 1\\
77.9591836734694 1\\
79.7959183673469 1\\
81.6326530612245 1\\
83.469387755102 1\\
85.3061224489796 1\\
87.1428571428571 1\\
88.9795918367347 1\\
90.8163265306122 1\\
92.6530612244898 1\\
94.4897959183673 1\\
96.3265306122449 1\\
98.1632653061225 1\\
100 1\\
};
\addlegendentry{\eqref{eq:Srelax}};

\addplot [
color=green!50!black,
solid,
line width=2.0pt
]
table[row sep=crcr]{
10 3\\
11.8367346938776 3\\
13.6734693877551 3\\
15.5102040816327 3\\
17.3469387755102 2\\
19.1836734693878 2\\
21.0204081632653 2\\
22.8571428571429 2\\
24.6938775510204 2\\
26.530612244898 2\\
28.3673469387755 2\\
30.2040816326531 2\\
32.0408163265306 2\\
33.8775510204082 2\\
35.7142857142857 2\\
37.5510204081633 2\\
39.3877551020408 2\\
41.2244897959184 2\\
43.0612244897959 2\\
44.8979591836735 2\\
46.734693877551 1\\
48.5714285714286 1\\
50.4081632653061 1\\
52.2448979591837 1\\
54.0816326530612 1\\
55.9183673469388 1\\
57.7551020408163 1\\
59.5918367346939 1\\
61.4285714285714 1\\
63.265306122449 1\\
65.1020408163265 1\\
66.9387755102041 1\\
68.7755102040816 1\\
70.6122448979592 1\\
72.4489795918367 1\\
74.2857142857143 1\\
76.1224489795918 1\\
77.9591836734694 1\\
79.7959183673469 1\\
81.6326530612245 1\\
83.469387755102 1\\
85.3061224489796 1\\
87.1428571428571 1\\
88.9795918367347 1\\
90.8163265306122 1\\
92.6530612244898 1\\
94.4897959183673 1\\
96.3265306122449 1\\
98.1632653061225 1\\
100 1\\
};
\addlegendentry{\eqref{eq:crossover}};

\addplot [
color=red,
solid,
line width=2.0pt
]
table[row sep=crcr]{
10 3\\
11.8367346938776 3\\
13.6734693877551 3\\
15.5102040816327 2\\
17.3469387755102 2\\
19.1836734693878 2\\
21.0204081632653 2\\
22.8571428571429 2\\
24.6938775510204 2\\
26.530612244898 2\\
28.3673469387755 2\\
30.2040816326531 2\\
32.0408163265306 2\\
33.8775510204082 2\\
35.7142857142857 2\\
37.5510204081633 2\\
39.3877551020408 2\\
41.2244897959184 2\\
43.0612244897959 2\\
44.8979591836735 2\\
46.734693877551 2\\
48.5714285714286 2\\
50.4081632653061 1\\
52.2448979591837 1\\
54.0816326530612 1\\
55.9183673469388 1\\
57.7551020408163 1\\
59.5918367346939 1\\
61.4285714285714 1\\
63.265306122449 1\\
65.1020408163265 1\\
66.9387755102041 1\\
68.7755102040816 1\\
70.6122448979592 1\\
72.4489795918367 1\\
74.2857142857143 1\\
76.1224489795918 1\\
77.9591836734694 1\\
79.7959183673469 1\\
81.6326530612245 1\\
83.469387755102 1\\
85.3061224489796 1\\
87.1428571428571 1\\
88.9795918367347 1\\
90.8163265306122 1\\
92.6530612244898 1\\
94.4897959183673 1\\
96.3265306122449 1\\
98.1632653061225 1\\
100 1\\
};
\addlegendentry{\eqref{eq:nuclearrelax}};

\end{axis}
\end{tikzpicture}%
} 
		\resizebox{\w}{!}{
%
%
\begin{tikzpicture}

\begin{axis}[%
width=4.52083333333333in,
height=3.565625in,
scale only axis,
xmin=10,
xmax=100,
ymin=1,
ymax=7,
legend style={draw=black,fill=white,legend cell align=left}
]
\addplot [
color=blue,
solid,
line width=2.0pt
]
table[row sep=crcr]{
10 7\\
11.8367346938776 5\\
13.6734693877551 5\\
15.5102040816327 4\\
17.3469387755102 4\\
19.1836734693878 4\\
21.0204081632653 4\\
22.8571428571429 4\\
24.6938775510204 3\\
26.530612244898 3\\
28.3673469387755 3\\
30.2040816326531 3\\
32.0408163265306 3\\
33.8775510204082 3\\
35.7142857142857 3\\
37.5510204081633 3\\
39.3877551020408 2\\
41.2244897959184 2\\
43.0612244897959 1\\
44.8979591836735 1\\
46.734693877551 1\\
48.5714285714286 1\\
50.4081632653061 1\\
52.2448979591837 1\\
54.0816326530612 1\\
55.9183673469388 1\\
57.7551020408163 1\\
59.5918367346939 1\\
61.4285714285714 1\\
63.265306122449 1\\
65.1020408163265 1\\
66.9387755102041 1\\
68.7755102040816 1\\
70.6122448979592 1\\
72.4489795918367 1\\
74.2857142857143 1\\
76.1224489795918 1\\
77.9591836734694 1\\
79.7959183673469 1\\
81.6326530612245 1\\
83.469387755102 1\\
85.3061224489796 1\\
87.1428571428571 1\\
88.9795918367347 1\\
90.8163265306122 1\\
92.6530612244898 1\\
94.4897959183673 1\\
96.3265306122449 1\\
98.1632653061225 1\\
100 1\\
};
\addlegendentry{\eqref{eq:Srelax}};

\addplot [
color=green!50!black,
solid,
line width=2.0pt
]
table[row sep=crcr]{
10 3\\
11.8367346938776 2\\
13.6734693877551 2\\
15.5102040816327 2\\
17.3469387755102 2\\
19.1836734693878 2\\
21.0204081632653 2\\
22.8571428571429 2\\
24.6938775510204 2\\
26.530612244898 2\\
28.3673469387755 2\\
30.2040816326531 2\\
32.0408163265306 2\\
33.8775510204082 2\\
35.7142857142857 2\\
37.5510204081633 2\\
39.3877551020408 1\\
41.2244897959184 1\\
43.0612244897959 1\\
44.8979591836735 1\\
46.734693877551 1\\
48.5714285714286 1\\
50.4081632653061 1\\
52.2448979591837 1\\
54.0816326530612 1\\
55.9183673469388 1\\
57.7551020408163 1\\
59.5918367346939 1\\
61.4285714285714 1\\
63.265306122449 1\\
65.1020408163265 1\\
66.9387755102041 1\\
68.7755102040816 1\\
70.6122448979592 1\\
72.4489795918367 1\\
74.2857142857143 1\\
76.1224489795918 1\\
77.9591836734694 1\\
79.7959183673469 1\\
81.6326530612245 1\\
83.469387755102 1\\
85.3061224489796 1\\
87.1428571428571 1\\
88.9795918367347 1\\
90.8163265306122 1\\
92.6530612244898 1\\
94.4897959183673 1\\
96.3265306122449 1\\
98.1632653061225 1\\
100 1\\
};
\addlegendentry{\eqref{eq:crossover}}

\addplot [
color=red,
solid,
line width=2.0pt
]
table[row sep=crcr]{
10 3\\
11.8367346938776 3\\
13.6734693877551 2\\
15.5102040816327 2\\
17.3469387755102 2\\
19.1836734693878 2\\
21.0204081632653 2\\
22.8571428571429 2\\
24.6938775510204 2\\
26.530612244898 1\\
28.3673469387755 1\\
30.2040816326531 1\\
32.0408163265306 1\\
33.8775510204082 1\\
35.7142857142857 1\\
37.5510204081633 1\\
39.3877551020408 1\\
41.2244897959184 1\\
43.0612244897959 1\\
44.8979591836735 1\\
46.734693877551 1\\
48.5714285714286 1\\
50.4081632653061 1\\
52.2448979591837 1\\
54.0816326530612 1\\
55.9183673469388 1\\
57.7551020408163 1\\
59.5918367346939 1\\
61.4285714285714 1\\
63.265306122449 1\\
65.1020408163265 1\\
66.9387755102041 1\\
68.7755102040816 1\\
70.6122448979592 1\\
72.4489795918367 1\\
74.2857142857143 1\\
76.1224489795918 1\\
77.9591836734694 1\\
79.7959183673469 1\\
81.6326530612245 1\\
83.469387755102 1\\
85.3061224489796 1\\
87.1428571428571 1\\
88.9795918367347 1\\
90.8163265306122 1\\
92.6530612244898 1\\
94.4897959183673 1\\
96.3265306122449 1\\
98.1632653061225 1\\
100 1\\
};
\addlegendentry{\eqref{eq:nuclearrelax}};

\end{axis}
\end{tikzpicture}%
}
	\end{center}
	\caption{Result of the four MOCAP experiments  (columns 1-4). Top: Regularization strength $\mu$ versus data fit $\|RX-M\|_F$. Middle: Regularization strength $\mu$ versus ground truth distance $\|X-X_{gt}\|_F$. Bottom: Regularization strength $\mu$ versus $\rank(X^\#)$.}
	\label{fig:deformdatafit}
\end{figure*}

In our final experiment we consider non rigid structure form motion with a rank prior.
We follow the aproach of Dai. et al. \cite{dai-etal-ijcv-2014} and let
\begin{equation}
X =
\left[
\begin{array}{c}
X_1 \\
Y_1 \\
Z_1 \\
\vdots \\
X_F \\
Y_F \\
Z_F \\
\end{array}
\right] \text{ and }
X^\# =
\left[
\begin{array}{ccc}
X_1 & Y_1 & Z_1 \\
\vdots & \vdots & \vdots \\
X_F &
Y_F &
Z_F
\end{array}
\right],
\end{equation}
where $X_i$,$Y_i$,$Z_i$ are $1 \times m$ matrices containing the $x$-,$y$- and $z$-coordinates of tracked image points in frame $i$.
With an orthographic camera the projection of the $3D$ points can be written $M = R X$,
where $R$ is a $2F \times 3F$ block diagonal matrix with $2 \times 3$ blocks $R_i$, consisting of two orthogonal rows that encode the camera orientation in image $i$. The resulting $2F \times m$ measurement matrix $M$ consists of the $x$- and $y$-image coordinates of the tracked points. Under the assumption of a linear shape basis model \cite{bregler-etal-cvpr-2000} with $r$ deformation modes, the matrix $X^\#$ can be written $X^\# = CB$, where $B$ is $r \times 3m$, and therefore $\rank(X^\#) = r$. We search for the matrix $X^\#$ of rank $r$ that minimizes the residual error $\|PX-M\|_F^2$.

The linear operator defined by $\A (X^\#) = RX$ does by itself not obey \eqref{eq:matrixRIP} since there are matrices of rank $1$ in its nullspace, see \cite{olsson-etal-iccv-2017}. 

In Figure~\ref{fig:deformdatafit} we compare the three relaxations 
\begin{eqnarray}
&\reg_{0,\mu}(\bfsigma(X^\#))+\|RX-M\|_F^2, &\label{eq:Srelax}\\
&\reg_{\mu,\lambda}(\bfsigma(X^\#))+ \|RX-M\|_F^2. \label{eq:crossover} &\\
& 2\sqrt{\mu}\|X^\#\|_*+\|RX-M\|_F^2, &\label{eq:nuclearrelax}
\end{eqnarray}
on the four MOCAP sequences displayed in Figure~\ref{fig:mocapshapes}, obtained from \cite{dai-etal-ijcv-2014}.
These consist of real motion capture data and therefore the ground truth solution is only approximatively of low rank.
Figure~\ref{fig:deformdatafit} shows results for the three methods.
We solved the problem for $50$ values of $\sqrt{\mu}$ between $10$ and $100$ (orange curve) and computed the resulting rank and datafit. (For \eqref{eq:crossover} we kept $\lambda=5$ fixed.)
All three formulations were given the same (random) starting solution.

The same tendencies are visible for all four sequences. While \eqref{eq:Srelax} generally gives a better data fit than \eqref{eq:nuclearrelax}, due to the nuclear norms shrinking bias, the distance to the ground truth is larger for low values of $\mu$ or equivalently large ranks where the problem gets ill posed. 
The relaxation \eqref{eq:crossover} consistently outperforms \eqref{eq:nuclearrelax} both in terms of data fit and distance to ground truth. In addition its performance is similar to \eqref{eq:Srelax} for high values of $\mu$ while it does not exhibit the same unstable behavior for high ranks.

{\small
\bibliographystyle{ieee}

}

\newpage

\appendix
\section{Proof of Theorem~\ref{thm:convenv}}\label{app:A}
\begin{proof}
	We need to prove that 
	\begin{equation}
	(f (\y) + \lambda \| \y\|_1 )^{**}=f^{**} (\y) + \lambda \| \y\|_1,
	\end{equation}
	where $*$ denotes the Fenchel conjugate, i.e.~ $g^*(\x)=\sup_{\y}\langle\x,\y\rangle-g(\y)$.
	For a general function $g$ the convex envelope can be computed using the bi-conjugate $g^{**}$.
	
	By symmetry it suffices to consider $\y\in\R^d_+$.
	First notice that in
	\begin{equation}
	( f(\cdot) + \lambda \| \cdot \|_1 )^* (\y)=\sup_{\x}  \langle \x,\y \rangle - (f(\x) + \lambda \| \x \|_1),
	\label{eq:conj1}
	\end{equation}
	only the term $\langle\x,\y\rangle$ depends on the signs of the elements of $x$. It is clear that any maximizing \( \x \) will have  \( \text{sign}(x_i)= \text{sign}(y_i) \) Therefore we may assume without loss of generality that $\x\in\R^d_+$ as well.
	We now have $\|\x\|_1 = \langle\x,\mathbf{1}\rangle$ which reduces \eqref{eq:conj1} to
	\begin{equation}
	\sup_{\x \in \mathbb{R}^d _+  } \langle \x,\y - \lambda \mathbf{1}  \rangle - f(\x).
	\end{equation}
	Note that if \( y_j - \lambda < 0  \)
	for some \(j \), then for every \( \x \in \mathbb{R}^d _+  \) we have \begin{equation}   \quad \ \langle \x - \mathbf{e}_j x_j, \y - \lambda \mathbf{1} \rangle - f( \x - x_j \mathbf{e}_j  )   \ge \langle \x , \y - \lambda \mathbf{1} \rangle - f( \x   ),
	\end{equation} where \(\mathbf{e}_j\) is the \(j\)th vector of the canonical basis. Therefore we introduce the set \( S= \{i \, : \, y_i < \lambda   \}  \) and the notation
	$\chi_{S^c} \x = \sum_{k \in S^c} \mathbf{e}_k x_k$.
	We then have
	\begin{equation}\begin{split}
	& \quad \ \sup_{\x \in \mathbb{R}^d _+  } \langle \x , \y - \lambda \mathbf{1} \rangle - f( \x   ) = \sup_{\x \in \mathbb{R}^d _+  } \langle \chi_{S^c} \x , \y - \lambda \mathbf{1} \rangle - f( \x   ) \\ &  = \sup_{\x \in \mathbb{R}^d  } \langle  \x , \chi_{S^c}(\y - \lambda \mathbf{1}) \rangle - f( \x   )  = f^* ( (\y - \lambda\mathbf{1})_+ )
	, \end{split}
	\end{equation}
	where $(\x)_+$ denotes thresholding at 0, that is, $(\x)_+ =  (\max(x_1,0),...,\max(x_d,0))$.
	
	To compute the second Fenchel-conjugate, first note that $f^*(\x+\mathbf{v})\geq f^*(\x)$ for \(\x, \mathbf{v} \in \mathbb{R}^d_+ \) since
	\begin{equation}\langle \y,\x \rangle - f(\y) - \|\y\|_1 \le \langle \y,\x + \mathbf{v} \rangle - f(\y) - \|\y\|_1  \end{equation} for all \( \y \in \mathbb{R}^d_+ \). Therefore we assume wlog \(\y \in \mathbb{R}^d_+ \). Note that the supremum in $\sup_{\x \in \mathbb{R}^d} \langle \x,\y \rangle - f^* ( (\x - \lambda\mathbf{1})_+ )$ is clearly  attained for an $\x$ with $x_j\geq \lambda$ for all $1\leq j\leq d$. By this observation we get $(f+\lambda\|\cdot\|_1)^{**}(\y)=$
	\begin{equation}
	\begin{split} &
	\sup_{\x \in \mathbb{R}^d} \langle \x,\y \rangle - f^* ( (\x - \lambda\mathbf{1})_+ )
	= \sup_{x_j\geq\lambda } \langle \x,\y \rangle - f^* ( \x - \lambda\mathbf{1} ) \\ & = \sup_{\mathbf{z} \in \mathbb{R}^d_+   } \langle \mathbf{z} + \lambda \mathbf{1},\y \rangle - f^* ( \mathbf{z} )    = \lambda \|\y\|_1 + \sup_{\mathbf{z} \in \mathbb{R}^d   } \langle \mathbf{z},\y \rangle - f^* ( \mathbf{z} ), \end{split}
	\end{equation}
	which shows that
	$(f+\lambda\|\cdot\|_1)^{**}(\y)= \lambda\|\y\|_1 + f^{**}(\y)$.
\end{proof}

Since the proof for the matrix case looks somewhat different we present it here.
It uses properties of the $\S_{\gamma}$-transform \cite{carlsson2016convexification} which is defined as
\begin{equation}
\S_{\gamma}(f)(\x)=(f+\frac{\gamma}{2}\|\cdot\|^2)^*(\gamma\x)-\frac{\gamma}{2}\|\x\|^2
\label{eq:Stransfdef}
\end{equation} 
The connection to $\Q_\gamma$ is that $\Q_\gamma(f)(\x) =\S_\gamma(\S_\gamma(f))(\x)$. Hence it is enought to show that two functions have the same $\S_\gamma$ transform to show that they have the same $\Q_\gamma$ transform.
\begin{proposition}\label{p2}
	Suppose that $f$ is a permutation and sign invariant $[0,\infty]$-valued functional on $\R^{d}$, $d=\min(n_1,n_2)$, and that $F(X)=f(\sigma(X)),$ $X\in\R^{n_1 \times n_2}$. Then $$\S_{\gamma}(F)(Y)=\S_{\gamma}(f)(\sigma(Y)).$$
\end{proposition}

\begin{proof}
	From \eqref{eq:Stransfdef} it can be seen (by completing squares) that
	\begin{equation}
	\begin{split}
	\S_{\gamma}(F)( Y) & =
	\sup_X -f(\sigma(X))-\frac{\gamma}{2}\fro{X-Y}_F.
	\end{split}
	\end{equation}
	Von Neumann's inequality (as stated e.g.~in \cite{carlsson2016convexification}) implies that the supremum is attained for an $X$ that shares singular vectors with $Y$. Hence
	\begin{equation}
	\S_{\gamma}(F)( Y)  =
	\left(\sup_{\nu_1\geq \nu_2 \geq \ldots } - f(\nu)-\frac{\gamma}{2}\fro{\nu-\sigma(Y)}\right),
	\label{eq:S1}
	\end{equation}
	Due to the permutation and sign invariance of $f$, we can drop the restrictions on $\nu$ which turns \eqref{eq:S1} into $\S_{\gamma}(f)(\sigma(Y))$.\qed
\end{proof}
For the case $F(X)=\mu\rank(X)_+\lambda\|X\|_*$ we get the expected result
\begin{equation}
\mathcal{S}_2^2 (\mu\rank( \cdot ) + \lambda \| \cdot \|_{1})(X)= \sum_i \mu - (\sqrt{\mu}-\sigma_i(X))_+^2 + \lambda\|X\|_*.
\end{equation}

\section{Minimizers of \eqref{eq:localS}}\label{app:B}
In this section we compute the minimizers of the one dimensional problem
\begin{equation}
\min_{x_i} \mu - (\max(\sqrt{\mu}-|x_i|,0))^2 +\lambda |x_i| + (x_i-z'_i)^2.
\end{equation} 
Because of symmetry of $(\max(\sqrt{\mu}-|x_i|,0))^2 +\lambda |x_i|$ it is clear that the optimal $x_i \geq 0$ if $z_i \geq 0$.  We therefore first consider 
\begin{equation}
\mu - (\max(\sqrt{\mu}-x_i,0))^2 +\lambda x_i + (x_i-z_i)^2,
\end{equation}
on $x_i \geq 0$. It is clear that this function is differentiable on $x_i > 0$ and goes to infinity when $x_i \rightarrow \infty$. Therefore the optimizer is either in a stationary point in $x_i > 0$ or in $x_i=0$.
Dividing into the two cases $0 < x_i < \sqrt{\mu}$ and $x_i \geq \sqrt{\mu}$ and differentiating shows that:
\begin{itemize}
	\item There are no stationary points if $z_i < \sqrt{\mu}+\frac{\lambda}{2}$. Thus the minimum is in $x_i = 0$.
	\item All points in $(0,\sqrt{\mu}]$ are stationary and have objective value $z^2_i$ if $z_i = \sqrt{\mu}+\frac{\lambda}{2}$. Since $x_i=0$ gives the same objective value
	any $x_i \in [0,\sqrt{\mu}]$ will be optimal.
	\item If $z_i > \sqrt{\mu}+\frac{\lambda}{2}$ then the point $x_i = z_i-\frac{\lambda}{2}$ is stationary with objective value
	\begin{equation}
	\mu+\lambda z_i - \frac{\lambda^2}{4} < \left(\sqrt{\mu}+\frac{\lambda}{2}\right)^2 < z^2_i.
	\end{equation}
	Since $x_i=0$ gives the objective $z_i^2$ the optimizer is $x_i = z_i-\frac{\lambda}{2}$.
\end{itemize}
Because of symmetry we now get that the minimizers of \eqref{eq:localS} is given by
This resulting minimizer is given by
\begin{equation}
x^*_i \in 
\begin{cases}
\{z_i -\sign(z_i)\frac{\lambda}{2}\} & |z_i| > \sqrt{\mu}+\frac{\lambda}{2} \\
[0, \sqrt{\mu}]\sign(z_i) & |z_i| = \sqrt{\mu}+\frac{\lambda}{2} \\
\{0\} & |z_i| < \sqrt{\mu}+\frac{\lambda}{2}
\end{cases}.
\end{equation}

\section{Proof of Lemma~\ref{lemma:subgradbnd}}\label{app:C}
\begin{proof}
	We first consider the scalar case $2z' \in \partial g_\lambda(x')$ where $z' \geq 0$.
	If $z' > \frac{\sqrt{\mu}}{1-\delta_K} + \frac{\lambda}{2}$ then in view of \eqref{eq:gsubdiff} we have $x'=z'-\frac{\lambda}{2}>\frac{\sqrt{\mu}}{1-\delta_K}$. 
	Now consider the linear function
	\begin{equation}
	l(x) = \delta_K(x-x')+z' = \delta_K x+(1-\delta_K) x'+\frac{\lambda}{2}.
	\end{equation}
	Since $l(x') = z'$ and $l(0) = (1-\delta_K) x'+\frac{\lambda}{2} > \sqrt{\mu}+\frac{\lambda}{2}$ it is clear from Figure~\ref{fig:gfun} that $l(x'') > z''$ for all $x''<x'$. Therefore 
	\begin{equation}
	z''-z' > l(x'')-z' = \delta_K(x''-x'),
	\end{equation}
	for all $x'' < x'$. Additionally, for $x'' > x'$ we clearly have
	\begin{equation}
	z''-z' = x''-x' > \delta_K(x''-x').
	\end{equation}
	Now assume that $2z' \in \partial g_\lambda(x')$ where $0\leq z' \leq (1-\delta_K) \sqrt{\mu} + \frac{\lambda}{2}$. Then because of \eqref{eq:gsubdiff} $x'=0$. We let $l(x) = \delta_K x+z'$. Since $l(0) = z' < \sqrt{\mu} + \frac{\lambda}{2}$ and 
	\begin{equation}
	l(\sqrt{\mu}) = \delta_K\sqrt{\mu}+z' <  \delta_K\sqrt{\mu}+(1-\delta_K) \sqrt{\mu} + \frac{\lambda}{2}
	\end{equation}
	it is clear that $l(x'') < z''$ for all $x'' > 0$. Therefore
	\begin{equation}
	z''-z' > l(x'')-z' = \delta_K x'' = \delta_K(x''-x').
	\end{equation}
	Similarly, it is easy to see that $l(x'') > z''$ if $x'' < 0$ and therefore
	\begin{equation}
	z'-z'' > z'-l(x'') = -\delta_K x'' = \delta_K (x'-x''),
	\end{equation}
	which shows that
	\begin{equation}
	(z''-z')(x''-x') > \delta_K (x''-x')^2.
	\end{equation}
	Because of symmetry of $\partial g_\lambda$ we conclude that the same holds if $z < 0$. To obtain \eqref{eq:subdiffest} we now sum over the non-zero entries of $\v$.
\end{proof}

We conclude this section by noting that the proof of the corresponding lemma for the matrix case is somewhat more complicated since the U and V matrices of the singular value decomposition of $X$ have to be accounted for. 
This can be done by combining the above result with Lemma 4.1 of \cite{olsson-etal-iccv-2017}.

\section{Proof of Proposition~\ref{prop:prox}}\label{app:D}

\begin{proof} It is enough to compute the proximal operator of the function \( \mathcal{S}^2 _\gamma (f) (\cdot) + \lambda \| \cdot \|_1  \). Wlog we assume that $\y\in \R_+^d$. With the same notation as in the proof of Theorem~\ref{thm:convenv} we have
	\begin{equation} \begin{split} & \quad \ \prox_{( \Q_\gamma (f)  + \lambda \| \cdot \|_1)/\rho} (\y)    \\
	& = \argmin_{\x \in \mathbb{R}^d} \frac{ \Q_\gamma (f) (\x)}{\rho} + \frac{\lambda}{\rho} \|\x\|_1 +  \| \x - \y \| ^2
	\\ & = \argmin_{\x \in \mathbb{R}^d _+}  \frac{ \Q_\gamma (f) (\x)}{\rho} +  \|\x\| ^2 - 2\langle \chi_{S^c} \x , \y - \frac{\lambda}{2\rho} \mathbf{1} \rangle + \|\y\| ^2
	\\ &= \argmin_{\x \in \mathbb{R}^d _+}  \frac{ \Q_\gamma (f) (\x)}{\rho} + \|\x\|^2 - 2\langle  \x , (\y - \frac{\lambda}{2\rho} \mathbf{1})_+ \rangle +  \|\y\|^2
	\end{split}   \end{equation}
	Since $\|\x\| ^2 - 2 \langle  \x , (\y - \frac{\lambda}{2\rho} \mathbf{1})_+ \rangle = \|\x-(\y - \frac{\lambda}{2\rho} \mathbf{1})_+ \|^2 -\|(\y - \frac{\lambda}{2\rho} \mathbf{1})_+ \|^2 $ and $\y$ is constant in the minimization of $\x$ we see that $\x$ also solves
	\begin{equation}
	\argmin_{\x \in \mathbb{R}^d _+}  \frac{ \Q_\gamma (f) (\x)}{\rho} + \|\x-(\y - \frac{\lambda}{2\rho} \mathbf{1})_+ \|^2.
	\end{equation}
	Note that $(\y - \frac{\lambda}{2\rho} \mathbf{1})_+ = \prox_{\lambda \|\cdot \|_1 / \rho}  (\y)$ since $\y \in \R_+^d$.
	Also, since the elements of $(\y - \frac{\lambda}{2\rho} \mathbf{1})_+$ are non-negative it is clear that minimizing over $\x \in \R^d$ instead of $\R^d_+$ does not change the optimizer and therefore
	\begin{equation}  \begin{split} & \quad \ \prox_{( \Q_\gamma (f)  + \lambda \| \cdot \|_1)/\rho} (\y)  = \prox_{ \Q_\gamma (f)/  \rho} (( \y -\frac{\lambda}{2\rho} \mathbf{1} )_+).    \end{split}  \end{equation}
\end{proof}

\end{document}